\documentclass[onefignum,onetabnum]{siamart190516}

\usepackage{lipsum}
\usepackage{amsmath,bm}
\usepackage{amsfonts}
\usepackage{graphicx}
\usepackage{wrapfig}
\usepackage{psfrag}
\usepackage{tikz}
\usetikzlibrary{calc,positioning}
\definecolor{corange}{rgb}{0.93, 0.57, 0.13}
\usepackage{simplewick}
\usepackage{mathrsfs}
\usepackage{url,hyperref}
\usepackage{setspace}
\usepackage{pifont}

\usepackage{algorithm}
\usepackage{algpseudocode}
\usepackage{amssymb}
\usepackage{booktabs}
\usepackage{comment}
\usepackage{subfigure}
\usepackage{float}
\usepackage{cite}
\usepackage{epstopdf}

\usepackage{tikz}

\allowdisplaybreaks[4]

\newtheorem{assumption}{Assumption}[section]

\def\d{\mathrm{d}}

\usepackage{lipsum}
\usepackage{amsfonts}
\usepackage{graphicx}
\usepackage{epstopdf}
\ifpdf
\DeclareGraphicsExtensions{.eps,.pdf,.png,.jpg}
\else
\DeclareGraphicsExtensions{.eps}
\fi

\newsiamremark{remark}{Remark}
\newsiamremark{hypothesis}{Hypothesis}
\crefname{hypothesis}{Hypothesis}{Hypotheses}
\newsiamthm{claim}{Claim}

\headers{Projected Walk on Spheres on Embedded Manifolds}{Z. Hui, C. Sheng, and  B. Su, and T. Zhou}

\title{A modified projected walk on spheres method for elliptic equations on high-dimensional embedded manifolds: algorithm and error estimates\thanks{Submitted to the editors DATE.
\funding{The work of Z. Hui is partially supported by the Fundamental Research Funds for the Central Universities (No. CXJJ-2024-437). The work of C. Sheng is supported by the National Natural Science Foundation of China (No. 12571389). The work of B. Su is supported by the National Natural Science Foundation of China (No. 12501541), the Scientific Research Start-up Funds of Hainan University (No. XJ2400010491). The work of T. Zhou is supported by the NSF of China (Nos. 12288201 and 12461160275). ${}^{\star}$Corresponding author. }
}}

\author{Zhiyuan Hui\thanks{School of Mathematics, Shanghai University of Finance and Economics, Shanghai 200433, China. Email: \email{202321\allowbreak3122@stu.sufe.edu.cn} (Z. Hui); \email{ctsheng@sufe.edu.cn} (C. Sheng).}
\and Changtao Sheng$^{\dagger,\star}$, 
Bihao Su\thanks{School of Mathematics and Statistics, Hainan University, Haikou 570100, China. Email: \email{bihaosu@hainanu.edu.\allowbreak cn} (B. Su).}
\and Tao Zhou\thanks{Institute of Computational Mathematics and Scientific/Engineering Computing, Academy of Mathematics and Systems Science, Chinese Academy of Sciences, Beijing 100190, China. Email: \email{tzhou@lsec.cc.ac.cn} (T. Zhou).}
}

\usepackage{amsopn}

\begin{document}
% \nolinenumbers

\maketitle

% REQUIRED
\begin{abstract}
In this paper, we propose a modified projected Walk on Spheres method (MPWoS) for screened Poisson equations on embedded manifolds. The method employs local extensions together with the Green representation in local Euclidean balls, coupled with a closest-point projection that maps the boundary samples back to the manifold. This formulation yields a meshfree and highly parallelizable stochastic recursion in the ambient Euclidean space, rather than a direct discretization of the Laplace--Beltrami operator on the manifold. The proposed approach can be viewed as a high-dimensional extension and modification of the projected Walk on Spheres method introduced for surface PDEs in [Sugimoto et al., SIGGRAPH Asia 2024 Conference Papers, pp. 1--10], with three main distinctions: a compensation term that corrects the discrepancy between the ambient Laplacian applied to the closest-point extension and the intrinsic Laplace--Beltrami operator on the manifold, an adaptive radius strategy determined by local geometric and boundary information, and a rigorous error analysis for the proposed algorithm. Under assumptions on the geometric projection and the prescribed compensation accuracy, we establish mean-square error estimates for the proposed Monte Carlo method in both the boundary and closed-manifold settings. Extensive numerical examples on parametrized, implicit, high-dimensional (up to 1000 dimensions), and point-cloud manifolds are presented to illustrate the convergence and efficiency of the proposed method across different geometries.
%In this paper, we propose a modified projected Walk on Spheres method (MPWoS) for screened Poisson equations on embedded manifolds. The method employs local extensions together with the Green representation in local Euclidean balls, coupled with a closest-point projection that maps the boundary samples back to the manifold. This formulation yields a meshfree and highly parallelizable stochastic recursion in the ambient Euclidean space, rather than a direct discretization of the Laplace--Beltrami operator on the manifold. To recover the intrinsic geometric structure of the problem, we introduce a compensation term for the discrepancy between the ambient Laplacian of the closest-point extension and the Laplace--Beltrami operator on the manifold, together with an adaptive radius strategy determined by local geometric and boundary information. Under the geometric projection and prescribed compensation accuracy, we establish mean-square error estimates for the proposed Monte Carlo method in both the boundary and closed-manifold settings. Extensive numerical examples on parametrized, implicit, high-dimensional (up to 1000 dimensions), and point-cloud manifolds are presented to illustrate the convergence and efficiency of the proposed method across different geometries.
\end{abstract}

% REQUIRED

% REQUIRED

\section{Introduction}

Partial differential equations posed on manifolds arise in many problems in the natural and applied sciences, including biological modeling, imaging, computer graphics, and fluid dynamics. Representative applications include bulk--surface transport and reaction--diffusion processes on curved membranes and tissues, where such equations are used to describe signaling, chemotaxis, and pattern formation \cite{MacdonaldMerrimanRuuth2013,MackenzieNolanInsall2016}; texture synthesis and brain surface mapping in imaging and computer graphics \cite{Turk1991,DiewaldPreuferRumpf2000,AngenentEtAl1999,GuEtAl2004}; as well as thin-film flow and surfactant transport on evolving interfaces in fluid dynamics and materials science \cite{MyersCharpinChapman2002,XuZhao2003}.

Over the past decades, a large variety of deterministic numerical methods have been developed for solving partial differential equations posed on surfaces and manifolds. Existing approaches may broadly be classified according to the geometric representation of the manifold, including parametrized or mesh-based manifolds, implicitly defined manifolds, and point clouds. For parametrized or mesh-based manifolds, one may work directly with local coordinates, coordinate transformations, or discrete surface meshes, leading to surface finite element, surface finite difference, and related intrinsic discretization methods (see, e.g., \cite{Dziuk2006,DziukElliott2013,Demlow2009,DziukElliott2007,BarreiraElliottMadzvamuse2011,HuJiaoKongWang2026,DziukElliott2007evolving} and the references therein). For implicitly defined manifolds, including those represented in level-set form, one may work directly in the ambient space, leading to embedding approaches such as implicit-surface formulations, closest-point methods, and unfitted or trace finite element methods (see, e.g., \cite{BertalmioChengOsherSapiro2001,RuuthMerriman2008,MacdonaldRuuth2008,MacdonaldRuuth2010,OlshanskiiReusken2018,GrandeLehrenfeldReusken2018,PetrasLingRuuth2018} and the references therein). For manifolds represented only by scattered samples, meshfree and point-cloud methods have also been developed based on local reconstruction, graph Laplacians, point integral methods, and generalized moving least-squares approximations (see, e.g., \cite{LiangZhao2013,BelkinNiyogi2008,PetronettoEtAl2013,LiShiSun2017,WangLeungZhao2018,GrossTraskKuberryAtzberger2020,JiangLiYanHarlim2024,LiYanJiang2026} and the references therein). Nevertheless, these methods still rely strongly on geometric representations and accurate approximations of geometric quantities and differential operators. Parametrization and mesh-based methods often require high-quality surface meshes or carefully constructed local coordinates; implicit methods depend on ambient-space or narrow-band discretizations together with accurate geometric descriptions; point-cloud methods typically require reliable local reconstruction or approximation of differential operators from scattered data. These difficulties become more pronounced in high-dimensional settings, where geometric representation, reconstruction, and global discretization are generally harder to perform efficiently, with the assembly and solution of global linear systems becoming increasingly costly for complex geometries or local pointwise computations.

Unlike the deterministic methods discussed above, stochastic algorithms provide an alternative approach for solving linear partial differential equations through the simulation of stochastic processes. Their theoretical foundation is the Feynman--Kac formula, which connects partial differential equations with probabilistic representations and enables the approximation of solutions through Monte Carlo sampling of stochastic paths. Based on this idea, stochastic algorithms have been widely developed for classical elliptic and parabolic equations, and exhibit advantages in high-dimensional problems, complex geometries, and local pointwise computations (cf. \cite{ShengSuXu2023}). Similar ideas can further be extended to PDEs posed on manifolds, where probabilistic representations may likewise be expressed in terms of Brownian motion on the manifold (see, e.g., \cite{Hsu2002}). Direct numerical simulation of manifold Brownian motion, however, is often nontrivial, since the stochastic process must remain compatible with the underlying geometry and typically requires intrinsic geometric information or suitable projection mechanisms. This motivates the development of stochastic methods that preserve the local and meshfree nature of Monte Carlo sampling while remaining suitable for computation on manifolds.

Among stochastic methods for elliptic equations in Euclidean domains, one of the earliest and most influential approaches is the Walk on Spheres (WoS) method introduced by Muller \cite{Muller1956} for the Dirichlet problem. Since then, Monte Carlo methods for elliptic equations have been systematically developed and studied (see \cite{Sabelfeld1991}). Early developments include the Monte Carlo method of DeLaurentis et al.\,\cite{DeLaurentisRomero1990} for local solutions of the Poisson equation with Dirichlet boundary conditions, while later work investigated the mathematical properties of WoS algorithms, including analysis of the $\varepsilon$-shell error in \cite{MascagniHwang2003}. More recently, WoS methods have been further developed with improved efficiency and extensions to screened Poisson equations (see, e.g., \cite{MillerSawhneyCraneGkioulekas2023,MillerSawhneyCraneGkioulekas2024} and the references therein). In addition, WoS ideas have also been extended to fractional PDEs (see, e.g., \cite{KyprianouOsojnikShardlow2018,ShengSuXu2023,HuiLiuShengXu2026}), which further illustrates the effectiveness of stochastic methods as local and meshfree solvers, particularly for pointwise evaluation and nonlocal problems. For PDEs posed on surfaces, however, WoS-type methods remain much less developed, with a recent contribution being the PWoS method of Sugimoto et al. \cite{Sugimoto2024}, which combines Monte Carlo sampling with the closest-point framework for solving surface PDEs. We also refer to related stochastic methods and probabilistic representations for PDEs on surfaces and manifolds \cite{Hsu2002,Guneysu2012}. Despite this progress, existing stochastic methods on manifolds remain relatively limited and often depend strongly on specific geometric representations or surface formulations, while rigorous convergence and error analysis for such stochastic algorithms still require further development.

The main purpose of this paper is to develop a modified projected Walk on Spheres method for solving the screened Poisson equation on embedded manifolds, together with a rigorous theoretical analysis of the proposed method. %Compared with the PWoS method for surface PDEs \cite{Sugimoto2024}, the present work further develops the framework in several directions. In particular, we derive a local probabilistic representation and the associated projected recursion for the screened Poisson equation on embedded manifolds, introduce a compensation mechanism, replace the medial axis based radius construction by a reach motivated adaptive radius strategy that is more interpretable and more suitable for high dimensional settings, extend the method from the surface setting to more general embedded manifolds, including parametrized, implicit, and point cloud manifolds as well as high dimensional cases, and establish mean square error estimates for the resulting Monte Carlo estimator. The main contributions and advantages of the proposed method can be summarized as follows.
To the best of our knowledge, several important questions remain open: 
i) Existing methods mainly focus on surface Laplace--Beltrami problems, while screened Poisson equations on more general embedded manifolds remain less explored.
ii) The discrepancy between the ambient Laplacian and the intrinsic manifold operator has not been systematically incorporated into the projected framework.
iii) The medial-axis- and local-feature-size-based radius construction may become computationally expensive and less transparent in high-dimensional settings.
iv) The extension to general embedded manifolds and high-dimensional settings still lacks a systematic theoretical framework. The main contributions and advantages of our new method can be summarized as follows.
\begin{itemize}

\item {\underline{Meshfree and parallelizable solver}:}
We derive a local Green representation for the screened Poisson equation on embedded manifolds and construct the associated MPWoS recursion. The resulting method is meshfree, avoids global linear systems, and is naturally suitable for parallel computation. Consequently, the method can efficiently handle high-dimensional problems (up to $1000$ dimensions), manifolds with complicated geometries, and point-cloud manifolds on a personal computer.

%\item \underline{Off-surface compensation:}
%We introduce an off-surface compensation term that accounts for the discrepancy between the ambient Laplacian acting on the closest-point extension and the intrinsic Laplace--Beltrami operator.
\item \underline{Extension-based compensation:}
We introduce a compensation term associated with the closest-point extension, which corrects the discrepancy between the ambient Laplacian applied to the extended function and the intrinsic Laplace--Beltrami operator on the embedded manifold.

\item \underline{Reach-based adaptive radius:}
%We propose a reach-motivated adaptive radius rule based on local boundary and geometric information, avoiding the direct computation of the medial axis or local feature size.
We propose a reach-motivated adaptive radius rule based only on local boundary and geometric information, avoiding the direct computation of the medial axis or local feature size, which is often impractical for high-dimensional embedded manifolds.

\item \underline{Error analysis:}
We prove mean-square error bounds for the proposed scheme under geometric projection and controlled compensation accuracy, and validate the method on parametrized, implicit, and point-cloud geometries.

\end{itemize}

The remainder of this paper is organized as follows. Section~2 introduces the geometric setting, the closest-point extension, and the projected Walk on Spheres algorithm for the screened Poisson equation on manifolds. Section~3 contains the error analysis of the proposed method. Numerical experiments are reported in Section~4 to demonstrate the performance of the method on manifolds of different dimensions and geometric complexity. Finally, concluding remarks are given in Section~5.

\vspace{-4pt}
\section{Modified projected Monte Carlo method for PDEs on embedded manifolds}\label{sec:algorithm}
This section derives the projected probabilistic representation and the associated Monte Carlo scheme. We first reformulate the screened Poisson equation using the closest-point extension, which was originally developed for surface PDEs in \cite{RuuthMerriman2008}. We then apply the Green representation in local Euclidean balls and project the resulting local boundary samples back to the manifold. This procedure yields a stopped stochastic recursion, which forms the basis of the subsequent error analysis.
\subsection{Problem setting and geometric framework}
Let $\mathcal M\subset\mathbb R^n$ be a connected compact embedded submanifold of dimension $d$. Throughout this paper, we consider the following two cases:
\begin{list}{$\diamond$}{%
\usecounter{enumi}
\setlength{\leftmargin}{3.0em}
\setlength{\labelwidth}{2.0em}
\setlength{\labelsep}{0.8em}
\setlength{\itemsep}{2pt}
\setlength{\topsep}{4pt}
}
\item Manifold-with-boundary case:
$\Omega\subset\mathcal M$ is a manifold patch with boundary
$\Gamma:=\partial\Omega\neq\emptyset$;

\item Closed-manifold case:
$\Omega=\mathcal M$ and $\partial\Omega=\emptyset$.
\end{list}
For clarity, we mainly focus on the manifold-with-boundary case and comment on the modifications for closed manifolds when appropriate.
We consider the screened Poisson equation
\begin{equation}\label{eq:surface_pde}
-\Delta_{\Omega}u_{\Omega}(y)+\sigma u_{\Omega}(y)=f_{\Omega}(y),\quad y\in\Omega.
\end{equation}
When $\Omega$ has boundary $\Gamma$ (i.e., {Case i}), it is supplemented with the Dirichlet boundary condition
\begin{equation}\label{eq:surface_bc}
u_{\Omega}(y)=g_{\Omega}(y),\quad y\in\Gamma.
\end{equation}
Here $\Delta_{\Omega}$ denotes the Laplace--Beltrami operator on $\Omega$, and $\sigma>0$ is a given constant. For a sufficiently smooth function $v$ on $\Omega$, it is given by $\Delta_\Omega v=\nabla_\Omega\cdot\nabla_\Omega v.$ 
We assume that $f_\Omega\in L^\infty(\Omega)$ is a given source term, and when $\Omega$ has boundary $\Gamma$, $g_\Omega\in L^\infty(\Gamma)$ is the prescribed Dirichlet datum.

%The proposed framework applies to the following three commonly used geometric representations:
%\vspace{-16pt}
%\begin{list}{(\roman{enumi})}{%
%\usecounter{enumi}
%\setlength{\leftmargin}{3.0em}
%\setlength{\labelwidth}{2.0em}
%\setlength{\labelsep}{0.8em}
%\setlength{\itemsep}{2pt}
%\setlength{\topsep}{4pt}
%}
%
%\item Parameterized manifold
%\begin{equation}\label{eq:parametric_manifold}
%\mathcal M=\{X(\xi):\xi\in U\subset\mathbb R^d\},
%\end{equation}
%where $X:U\to\mathbb R^n$ is a sufficiently smooth map;
%
%\item Implicitly represented manifold
%\begin{equation}\label{eq:implicit_manifold}
%\mathcal M=\{x\in\mathbb R^n:\Phi(x)=0\},
%\qquad
%\Phi=(\phi_1,\dots,\phi_{n-d}),
%\end{equation}
%with $\nabla\phi_1,\dots,\nabla\phi_{n-d}$ linearly independent on $\mathcal M$; in the codimension-one case, this reduces to the level-set form
%\begin{equation}\label{eq:levelset_manifold}
%\mathcal M=\{x\in\mathbb R^n:\phi(x)=0\};
%\end{equation}
%
%\item Manifold represented by a point cloud
%\begin{equation}\label{eq:pointcloud_manifold}
%\mathcal P=\{p_i\}_{i=1}^N\subset\mathbb R^n.
%\end{equation}
%
%\end{list}
The proposed framework applies to the following three commonly used geometric representations:
\vspace{-3pt}
\begin{list}{$\circ$}{%
\setlength{\leftmargin}{3.0em}
\setlength{\labelwidth}{1.5em}
\setlength{\labelsep}{0.8em}
\setlength{\itemsep}{2pt}
\setlength{\topsep}{4pt}
}

\item Parameterized manifold
%\begin{equation}\label{eq:parametric_manifold}
$\mathcal M=\{X(\xi):\xi\in U\subset\mathbb R^d\},$ 
%\end{equation}
where $X:U\to\mathbb R^n$ is a sufficiently smooth map;

\item Implicitly represented manifold
%\begin{equation}\label{eq:implicit_manifold}
$\mathcal M=\{x\in\mathbb R^n:\Phi(x)=0\},$
%\qquad
$\Phi=(\phi_1,\dots,\phi_{n-d}),$
%\end{equation}
with $\nabla\phi_1,\dots,\nabla\phi_{n-d}$ linearly independent on $\mathcal M$; in the codimension-one case, this reduces to the level-set form
%\begin{equation}\label{eq:levelset_manifold}
$\mathcal M=\{x\in\mathbb R^n:\phi(x)=0\};$
%\end{equation}

\item Manifold represented by a point cloud
%\begin{equation}\label{eq:pointcloud_manifold}
$\mathcal P=\{p_i\}_{i=1}^N\subset\mathbb R^n.$
%\end{equation}

\end{list}

Since the subsequent derivation is based on local sampling in Euclidean space, we work in a tubular neighborhood of $\mathcal M$ and employ the closest-point extension to formulate the projected recursion. Accordingly, we assume that $\mathcal M$ and $\Omega$ are sufficiently smooth so that the closest-point map is well defined and smooth in a tubular neighborhood of $\mathcal M$. To distinguish intrinsic points on the manifold from points in the embedding space, we use $y$ to denote points on $\mathcal M$, while $x$ denotes generic points in $\mathbb R^n$. Assume that there exists $\delta>0$, chosen below the reach of $\mathcal M$, such that
\begin{equation}
U_{\delta}:=\big\{x\in\mathbb{R}^{n}:\operatorname{dist}(x,\mathcal{M})<\delta\big\}
\end{equation}
is a tubular neighborhood of $\mathcal M$, so that every $x\in U_\delta$ has a unique nearest point on $\mathcal M$. The associated closest-point map
\begin{equation}\label{eq:cp_map}
\operatorname{cp}:U_{\delta}\to\mathcal{M},
\quad
\operatorname{cp}(x):=\arg\min_{y\in\mathcal M}|x-y|,
\end{equation}
is assumed to be sufficiently smooth. For each $y\in\mathcal M$, let $T_y\mathcal M$ denote the tangent space at $y$ and
$N_y\mathcal M:=(T_y\mathcal M)^\perp\subset\mathbb R^n$ the normal space.
Then every $x\in U_\delta$ admits the representation
% \begin{equation}\label{eq:normal_representation}
% x=\operatorname{cp}(x)+\eta(x),
% \quad
% \eta(x)=\operatorname{dist}(x,\mathcal M)\in N_{\operatorname{cp}(x)}\mathcal M.
% \end{equation}
\begin{equation}\label{eq:normal_representation}
x=\operatorname{cp}(x)+\eta(x),
\quad
\eta(x):=x-\operatorname{cp}(x)\in
N_{\operatorname{cp}(x)}\mathcal M,
\end{equation}
with $|\eta(x)|=\operatorname{dist}(x,\mathcal M)$.
%with $|\eta(x)|=\operatorname{dist}(x,\mathcal M)$.

To formulate the projected recursion in the ambient Euclidean space, we extend functions on $\Omega$ to the tubular neighborhood via $\operatorname{cp}$. Define $U_{\delta}^{\Omega}:=\{x\in U_{\delta}:\operatorname{cp}(x)\in\Omega\},$
and, for any function $v$ on $\Omega$, set
\begin{equation}\label{eq:cp_extension}
Ev(x):=v(\operatorname{cp}(x)),\quad x\in U_{\delta}^{\Omega}.
\end{equation}
By \eqref{eq:normal_representation}, the closest-point extension is constant along normal fibers of $\mathcal M$. In particular,
$$u(x):=Eu_{\Omega}(x)=u_{\Omega}(\operatorname{cp}(x)),\quad f(x):=Ef_{\Omega}(x)=f_{\Omega}(\operatorname{cp}(x)),$$
and, when $\Omega$ has boundary $\Gamma$, $g(x):=Eg_{\Omega}(x)=g_{\Omega}(\operatorname{cp}(x)).$
Then, one has the standard identity
\begin{equation}\label{eq:LB_cp_identity}
\Delta_{\Omega}u_{\Omega}(y)=\Delta u(y),\quad y\in\Omega,
\end{equation}
where $\Delta$ denotes the Euclidean Laplacian in $\mathbb R^n$. However, away from the manifold, the ambient operator acting on the closest-point extension generally does not reproduce the intrinsic surface operator. We therefore introduce the compensation term, defined by
\begin{equation}\label{eq:mismatch_definition}
m(x):=-\Delta u(x)+\sigma u(x)-f(x),\quad x\in U_\delta^\Omega.
\end{equation}
It follows from \eqref{eq:LB_cp_identity} that $m(y)=0$ for all $y\in\Omega$.  

Thus the closest-point extension of \eqref{eq:surface_pde} and \eqref{eq:surface_bc} satisfies the compensated ambient equation
\begin{equation}\label{eq:embedding_problem}
-\Delta u(x)+\sigma u(x)=f(x)+m(x),
\quad x\in U_\delta^\Omega ,
\end{equation}
where, in the case $\Gamma\neq\emptyset$, the corresponding extended boundary is given by
\begin{equation}\label{eq:Gamma_ext_definition}
u(x)=g(x),\quad x\in\Gamma^{\mathrm{ext}}:=\big\{x\in U_\delta^\Omega:\operatorname{cp}(x)\in\Gamma\big\}.
\end{equation}
Note that, in the closed manifold case, we set $\Gamma^{\mathrm{ext}}=\emptyset$.
Based on the embedding-space equation \eqref{eq:embedding_problem} and the extended boundary \eqref{eq:Gamma_ext_definition}, we next proceed with the theoretical derivation of the MPWoS method.

%For notational convenience, in the case of manifolds with boundary we define the extended boundary by
%\begin{equation}\label{eq:Gamma_ext_definition}
%\Gamma^{\mathrm{ext}}
%:=
%\{x\in U_\delta^\Omega:\operatorname{cp}(x)\in\Gamma\},
%\end{equation}
%while in the closed manifold case we set $\Gamma^{\mathrm{ext}}=\emptyset$.

%When $\Omega$ has boundary $\Gamma$, let $\delta_\Gamma(y):=\operatorname{dist}(y,\Gamma)$ for $y$ near $\Gamma$. We assume that $u_\Omega$ is Lipschitz continuous near $\Gamma$, so that there exists a constant $L_\Gamma>0$ such that
%\begin{equation}\label{u_lip}
%    |u_\Omega(y)-u_\Omega(\operatorname{cp}_\Gamma(y))|
%\le
%L_\Gamma\,\delta_\Gamma(y),
%\qquad y\in\Gamma_\varepsilon.
%\end{equation}
%
%The screened Poisson equation in the embedding space takes the form
%\begin{equation}\label{eq:embedding_problem}
%-\Delta u(x)+\sigma u(x)=f(x)+m(x),
%\qquad x\in U_\delta^\Omega\setminus\Gamma^{\mathrm{ext}}.
%\end{equation}
%When $\Omega$ has boundary $\Gamma$, it is supplemented with the boundary condition
%\begin{equation}\label{eq:embedding_bc}
%u(x)=g(x),
%\qquad x\in\Gamma^{\mathrm{ext}}.
%\end{equation}
%A suitable approximation of the compensation term $m$ will be specified later and used in the projected Walk on Spheres construction.
%With these geometric ingredients and the reformulated equation in the embedding space at hand, we are now in a position to introduce the projected Walk on Spheres recursion.

\subsection{Local Green representation and projected recursion}
In this subsection, we derive the conditional-expectation representation and projected recursion underlying the reformulated screened Poisson equation. The derivation is based on a local Green representation in Euclidean balls within the tubular neighborhood, combined with closest-point projection of the resulting boundary samples back to the manifold. 

For each interior point $y\in\Omega$, let $\mathbb{B}^{n}_{r(y)}(y):=\{x\in\mathbb{R}^{n}:|x-y|<r(y)\}$ be a Euclidean ball centered at $y$ with radius $r(y)$ satisfying
\begin{equation}\label{eq:admissible_ball}
\mathbb{B}_{r(y)}^{n}(y)\subset U_\delta^\Omega\; \;\text{and}\;\;\;\mathbb{B}_{r(y)}^{n}(y)\cap \Gamma^{\mathrm{ext}}=\emptyset\quad\text{when }\Gamma\neq\emptyset.
\end{equation}

The precise choice of $r(y)$ will be specified in Section~\ref{sec:radius_selection}. At the present stage, it suffices to assume that the admissible ball remains inside the tubular neighborhood and, when $\Gamma\neq\emptyset$, stays away from the extended boundary.

We begin with a local representation formula in the ball $\mathbb{B}_{r(y)}^{n}(y)$.

\begin{lemma}\label{lem:local_ball_representation}
Let $y$ be an interior point of $\Omega$, and suppose that \eqref{eq:admissible_ball} holds. Then the solution of \eqref{eq:embedding_problem}, together with \eqref{eq:Gamma_ext_definition}, admits the integral representation
\begin{equation}\label{eq:local_ball_representation}
u(y)
=
c_{\sigma,r(y)}
\frac{1}{|\partial \mathbb{B}_{r(y)}^{n}(y)|}
\int_{\partial \mathbb{B}_{r(y)}^{n}(y)} u(z)\,\d S(z)
+
\int_{\mathbb{B}_{r(y)}^{n}(y)} \mathcal{G}(y,z)\,\bigl(f(z)+m(z)\bigr)\,\d z,
\end{equation}
where $\mathcal{G}(y,z)$ denotes the Dirichlet Green function associated with the operator $-\Delta+\sigma$ in the ball $\mathbb{B}_{r(y)}^{n}(y)$ evaluated at the center point $y$, and the constant 
\begin{equation}\label{eq:def_c_sigma_r}
c_{\sigma,r}
=
\frac{(\sqrt{\sigma}\,r)^{\nu}}{2^{\nu}\Gamma(\nu+1)\,I_{\nu}(\sqrt{\sigma}\,r)},
\quad
\nu=\frac{n-2}{2}.
\end{equation}
Here $I_{\nu}$ denotes the modified Bessel function of the first kind.
\end{lemma}
\begin{proof}
Let $\mathbb{B}_r:=\mathbb{B}_r^{n}(y)$ for simplicity in this proof. By the Green representation formula for the operator $-\Delta+\sigma$ in $\mathbb{B}_r$, we have
\[
u(y)
=
\int_{\partial \mathbb{B}_r} P_r(y,z)\,u(z)\,\d S(z)
+
\int_{\mathbb{B}_r} \mathcal{G}(y,z)\,\bigl(f(z)+m(z)\bigr)\,\d z.
\]
Since the observation point is the center of the ball and both the ball and the operator $-\Delta+\sigma$ are rotationally invariant, the boundary kernel is constant on $\partial \mathbb{B}_r$. Thus there exists a constant $p_{\sigma,r}$ such that $P_r(y,z)=p_{\sigma,r}$ for $z\in \partial \mathbb{B}_r$. Hence
\[
\int_{\partial \mathbb{B}_r} P_r(y,z)\,u(z)\,\d S(z)
=
p_{\sigma,r}\int_{\partial \mathbb{B}_r}u(z)\,\d S(z)
=
c_{\sigma,r}\frac{1}{|\partial \mathbb{B}_r|}
\int_{\partial \mathbb{B}_r}u(z)\,\d S(z),
\]
where $c_{\sigma,r}:=p_{\sigma,r}\,|\partial \mathbb{B}_r|$. This gives \eqref{eq:local_ball_representation}.

It remains to determine $c_{\sigma,r}$. Let $v$ solve $-\Delta v+\sigma v=0$ in $\mathbb B_r$ with $v=1$ on $\partial\mathbb B_r$.
% Applying Lemma \ref{lem:local_ball_representation} to $v$, and using $v=1$ on $\partial \mathbb{B}_r$, we obtain $v(y)=c_{\sigma,r}$. 
Applying the Green representation formula \eqref{eq:local_ball_representation} to $v$, and using $v=1$ on $\partial \mathbb{B}_r$, we obtain $v(y)=c_{\sigma,r}$.
By radial symmetry, $v(x)=V(\rho)$ with $\rho=|x-y|$, where
\[
V''(\rho)+\frac{n-1}{\rho}V'(\rho)-\sigma V(\rho)=0,
\quad 0<\rho<r,
\]
subject to $V(r)=1$ and $V$ bounded as $\rho\to 0$. Setting $\nu=\frac{n-2}{2}$, the bounded radial solution has the form $V(\rho)=C\,\rho^{-\nu}I_\nu(\sqrt{\sigma}\,\rho).$
From the boundary condition $V(r)=1$, we obtain
$C=\frac{r^\nu}{I_\nu(\sqrt{\sigma}\,r)}$, so that
\[
V(\rho)=\frac{r^\nu}{I_\nu(\sqrt{\sigma}\,r)}\,\rho^{-\nu}I_\nu(\sqrt{\sigma}\,\rho).
\]
Using $I_\nu(t)\sim \frac{1}{\Gamma(\nu+1)}\left(\frac{t}{2}\right)^\nu$ as $t\to 0$, and noting that $y$ is the center of the ball, we conclude that
\[
c_{\sigma,r}=v(y)=V(0)=\frac{(\sqrt{\sigma}\,r)^\nu}{2^\nu\Gamma(\nu+1)\,I_\nu(\sqrt{\sigma}\,r)}.
\]
This ends the proof.
\end{proof}

\begin{remark}
\itshape
Let $\rho:=|y-z|$ and $\lambda:=\sqrt{\sigma}$. The Green function in \eqref{eq:local_ball_representation} is the Dirichlet Green function associated with the operator $-\Delta+\sigma$ in the ball $\mathbb{B}_{r(y)}^{n}(y)$. Since the pole is located at the center $y$, this Green function is radial and admits the following explicit representation in terms of modified Bessel functions:
\[
\mathcal{G}(y,z)
=
\frac{1}{(2\pi)^{n/2}}
\Big(\frac{\lambda}{\rho}\Big)^{\frac{n-2}{2}}
\bigg[
K_{\frac{n-2}{2}}(\lambda\rho)
-
\frac{K_{\frac{n-2}{2}}(\lambda r(y))}{I_{\frac{n-2}{2}}(\lambda r(y))}
I_{\frac{n-2}{2}}(\lambda\rho)
\bigg].
\]
In particular, for the cases $n=2$ and $n=3$, the corresponding formulas can be found in {\rm\cite[Appendix~A]{Sawhney2023}}:
\[
\begin{aligned}
\mathcal{G}(y,z)
&=
\frac{1}{2\pi}
\bigg[K_0(\lambda\rho)-\frac{K_0(\lambda r(y))}{I_0(\lambda r(y))}I_0(\lambda\rho)\bigg],
&&\text{when } n=2,\\[6pt]
\mathcal{G}(y,z)
&=
\frac{1}{4\pi\rho}
\frac{\sinh\bigl(\lambda(r(y)-\rho)\bigr)}{\sinh(\lambda r(y))},
&&\text{when } n=3.
\end{aligned}
\]
Here $I_\nu$ and $K_\nu$ denote the modified Bessel functions of the first and second kind, respectively.
\end{remark}

To derive a recursive conditional expectation representation, we introduce
\begin{equation}\label{eq:omega_of_y}
\rho(y,z):=\frac{\mathcal{G}(y,z)}{\omega(y)},\quad\omega(y):=\int_{\mathbb{B}_{r(y)}^{n}(y)} \mathcal{G}(y,z)\,\d z,
\quad y\in\Omega,\;\;z\in \mathbb{B}_{r(y)}^{n}(y).
\end{equation}
%and define
%\begin{equation}\label{eq:rho_density}
%\rho(y,z):=\frac{\mathcal{G}(y,z)}{\omega(y)},
%\qquad 
%\end{equation}
Then $\rho(y,\cdot)$ is a probability density on $\mathbb{B}_{r(y)}^{n}(y)$.
Starting from $Y_0=y_0\in\Omega$, we construct the projected chain $\{Y_k\}_{k\ge0}$ recursively as follows. Conditionally on $Y_k$, let $\Theta_{k+1}$ be uniformly distributed on $\partial\mathbb B_{r(Y_k)}^n(Y_k)$, while $Z_{k+1}$ has density $\rho(Y_k,\cdot)$ in $\mathbb B_{r(Y_k)}^n(Y_k)$. The random variables $\Theta_{k+1}$ and $Z_{k+1}$ are assumed to be conditionally independent given $Y_k$. The next projected state is then defined by
%\begin{equation}\label{eq:projected_chain_definition}
$Y_{k+1}=\operatorname{cp}(\Theta_{k+1}).$ 
 With the above notation, we can derive the final representation formulas for both the manifold-with-boundary case and the closed manifold case.

\begin{theorem}\label{thm:final_representation}Let $y_0\in\Omega$ be an interior point, and let $\{Y_k\}_{k\ge0}$ be the projected chain defined by $Y_{k+1}=\operatorname{cp}(\Theta_{k+1})$, where the closest-point map $\operatorname{cp}(\cdot)$ defined in \eqref{eq:cp_map}.
%Assume that \eqref{eq:n_step_representation} holds.

\medskip
\noindent\textnormal{(i)} Manifold with boundary case, i.e., $\Gamma\neq\emptyset$:
Assume that the projected chain reaches the boundary at an almost surely finite index $K$, namely, $Y_K\in\Gamma.$ Then
\begin{equation}\label{eq:boundary_representation}
u_\Omega(y_0)
=
\mathbb E\bigl[W_K\,g_\Omega(Y_K)\bigr]
+
\mathbb E\bigg[\,\sum_{k=0}^{K-1}
W_k\,\omega(Y_k)\bigl(f(Z_{k+1})+m(Z_{k+1})\bigr)
\bigg].
\end{equation}
% where $$W_0:=1,\quad W_k:=\prod_{j=0}^{k-1} c_{\sigma,r(Y_j)},\quad k\ge1,\quad \omega(y)=\frac{1-c_{\sigma,r(y)}}{\sigma}$$
% and $c_{\sigma,r}$ is given in \eqref{eq:def_c_sigma_r}.

\medskip
\noindent\textnormal{(ii)} Closed manifold case, i.e., $\Gamma=\emptyset$:
Assume  that $\lim_{n\to\infty}\mathbb E\bigl[W_n\,u_\Omega(Y_n)\bigr]=0$, then
\begin{equation}\label{eq:boundaryless_representation}
u_\Omega(y_0)
=
\mathbb E\bigg[\,\sum_{k=0}^{\infty}W_k\,\omega(Y_k)\bigl(f(Z_{k+1})+m(Z_{k+1})\bigr)\bigg].
\end{equation}
In the above, the constants is given by
\begin{equation}\label{eq:omega_explicit_formula}
W_0:=1,
\quad
W_k:=\prod_{j=0}^{k-1}c_{\sigma,r(Y_j)},
\quad k\ge1,\quad \omega(y)=\frac{1-c_{\sigma,r(y)}}{\sigma},
\end{equation}
where $c_{\sigma,r}$ is given in \eqref{eq:def_c_sigma_r}.
\end{theorem}

\begin{proof}
% The expression for $\rho(y,z)$ follows immediately from its definition.
We first establish the finite-step identity from which both representation formulas follow. Using the local representation formula in Lemma~\ref{lem:local_ball_representation} together with the definition of $\rho(y,\cdot)$, we obtain
\begin{equation*}
u(y)
=
c_{\sigma,r(y)}
\frac{1}{|\partial \mathbb{B}_{r(y)}^{n}(y)|}
\int_{\partial \mathbb{B}_{r(y)}^{n}(y)} u(z)\,\d S(z)
+
\int_{\mathbb{B}_{r(y)}^{n}(y)}
\mathcal{G}(y,z)\bigl(f(z)+m(z)\bigr)\,\d z.
\end{equation*}
Since $\Theta_y$ is distributed according to the normalized surface measure on $\partial \mathbb{B}_{r(y)}^{n}(y)$, we have
\begin{equation*}
\mathbb E\bigl[u(\Theta_y)\bigr]
=
\frac{1}{|\partial \mathbb{B}_{r(y)}^{n}(y)|}
\int_{\partial \mathbb{B}_{r(y)}^{n}(y)} u(z)\,\d S(z).
\end{equation*}
Moreover, by the definition of the density $\rho(y,\cdot)$,
\begin{equation*}
\begin{split}
\mathbb E\bigl[f(Z_y)+m(Z_y)\bigr]
&=
\int_{\mathbb{B}_{r(y)}^{n}(y)}
\bigl(f(z)+m(z)\bigr)\rho(y,z)\,\d z =
\frac{1}{\omega(y)}
\int_{\mathbb{B}_{r(y)}^{n}(y)}
\mathcal{G}(y,z)\bigl(f(z)+m(z)\bigr)\,\d z.
\end{split}
\end{equation*}
Combining the above identities yields
\begin{equation*}
u(y)
=
c_{\sigma,r(y)}\,\mathbb E\bigl[u(\Theta_y)\bigr]
+
\omega(y)\,\mathbb E\bigl[f(Z_y)+m(Z_y)\bigr].
\end{equation*}
Since $u$ is the closest-point extension of $u_\Omega$, it follows that $u(\Theta_y)=u_\Omega(\operatorname{cp}(\Theta_y))$. Therefore,
\begin{equation*}
u_\Omega(y)
=
c_{\sigma,r(y)}\,\mathbb E\bigl[u_\Omega(\operatorname{cp}(\Theta_y))\bigr]
+
\omega(y)\,\mathbb E\bigl[f(Z_y)+m(Z_y)\bigr].
\end{equation*}
Substituting $y=Y_k$ into the above identity gives the one-step relation
\begin{equation}\label{eq:one_step_markov}
\begin{split}
u_\Omega(Y_k)
=\;&
c_{\sigma,r(Y_k)}
\,\mathbb E\bigl[u_\Omega(Y_{k+1})\mid Y_k\bigr]
+
\omega(Y_k)
\,\mathbb E\bigl[f(Z_{k+1})+m(Z_{k+1})\mid Y_k\bigr],
\qquad k\ge0.
\end{split}
\end{equation}
Writing $W_0:=1$ and $W_{k+1}:=W_kc_{\sigma,r(Y_k)}$, multiplying \eqref{eq:one_step_markov} by $W_k$ gives
\begin{equation*}
W_k\,u_\Omega(Y_k)
=
\mathbb E\bigl[W_{k+1}u_\Omega(Y_{k+1})\mid Y_k\bigr]
+
W_k\,\omega(Y_k)
\,\mathbb E\bigl[f(Z_{k+1})+m(Z_{k+1})\mid Y_k\bigr].
\end{equation*}
Taking expectations on both sides and summing from $k=0$ to $N_s-1$ gives
\begin{equation}\label{eq:n_step_representation}
u_\Omega(y_0)
=
\mathbb E\bigl[W_{N_s}\,u_\Omega(Y_{N_s})\bigr]
+
\sum_{k=0}^{N_s-1}
\mathbb E\Bigl[
W_k\,\omega(Y_k)\bigl(f(Z_{k+1})+m(Z_{k+1})\bigr)
\Bigr].
\end{equation}

For case \textnormal{(i)}, we apply \eqref{eq:n_step_representation} to the stopped chain. Passing to the stopping index $K$ yields
\begin{equation*}
u_\Omega(y_0)
=
\mathbb E\bigl[W_K\,u_\Omega(Y_K)\bigr]
+
\mathbb E\bigg[
\sum_{k=0}^{K-1}
W_k\,\omega(Y_k)\bigl(f(Z_{k+1})+m(Z_{k+1})\bigr)
\bigg].
\end{equation*}
Since $Y_K\in\Gamma$ and $u_\Omega=g_\Omega$ on $\Gamma$, we have $u_\Omega(Y_K)=g_\Omega(Y_K)$.
Substituting this identity into the preceding formula gives \eqref{eq:boundary_representation}.
For case \textnormal{(ii)}, letting $N_s\to\infty$ in \eqref{eq:n_step_representation} and using the assumption $\mathbb E\bigl[W_{N_s}\,u_\Omega(Y_{N_s})\bigr]\to0$, we obtain \eqref{eq:boundaryless_representation}.

Finally, we turn to the explicit form of $\omega(y)$. Fix $y\in\Omega$, let $q$ solve $-\Delta q+\sigma q=1$ in $\mathbb B_{r(y)}^n(y)$ with homogeneous Dirichlet boundary condition. By Lemma~\ref{lem:local_ball_representation}, it is clear that
$q(y)=\int_{\mathbb B_{r(y)}^n(y)}\mathcal G(y,z)\,{\rm d}z=\omega(y).$ 
Moreover, let $v$ solve $-\Delta v+\sigma v=0$ in $\mathbb B_{r(y)}^n(y)$ with $v=1$ on $\partial\mathbb \mathbb B_{r(y)}^n(y)$. Since $v(y)=c_{\sigma,r}$ and $q=(1-v)/\sigma$ by uniqueness, $\omega(y)=q(y)=\frac{1-c_{\sigma,r}}{\sigma}$. 
This ends the proof.
\end{proof}

\subsection{Efficient algorithm and implementation}

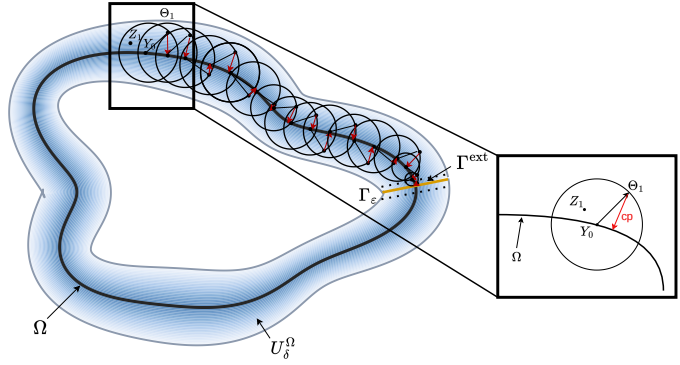
\begin{figure}[ht!]
\centering
\resizebox{0.68\textwidth}{!}{%
\begin{tikzpicture}

% Main figure
\node[anchor=south west,inner sep=0] (main) at (0,0)
{\includegraphics[width=6.5cm]{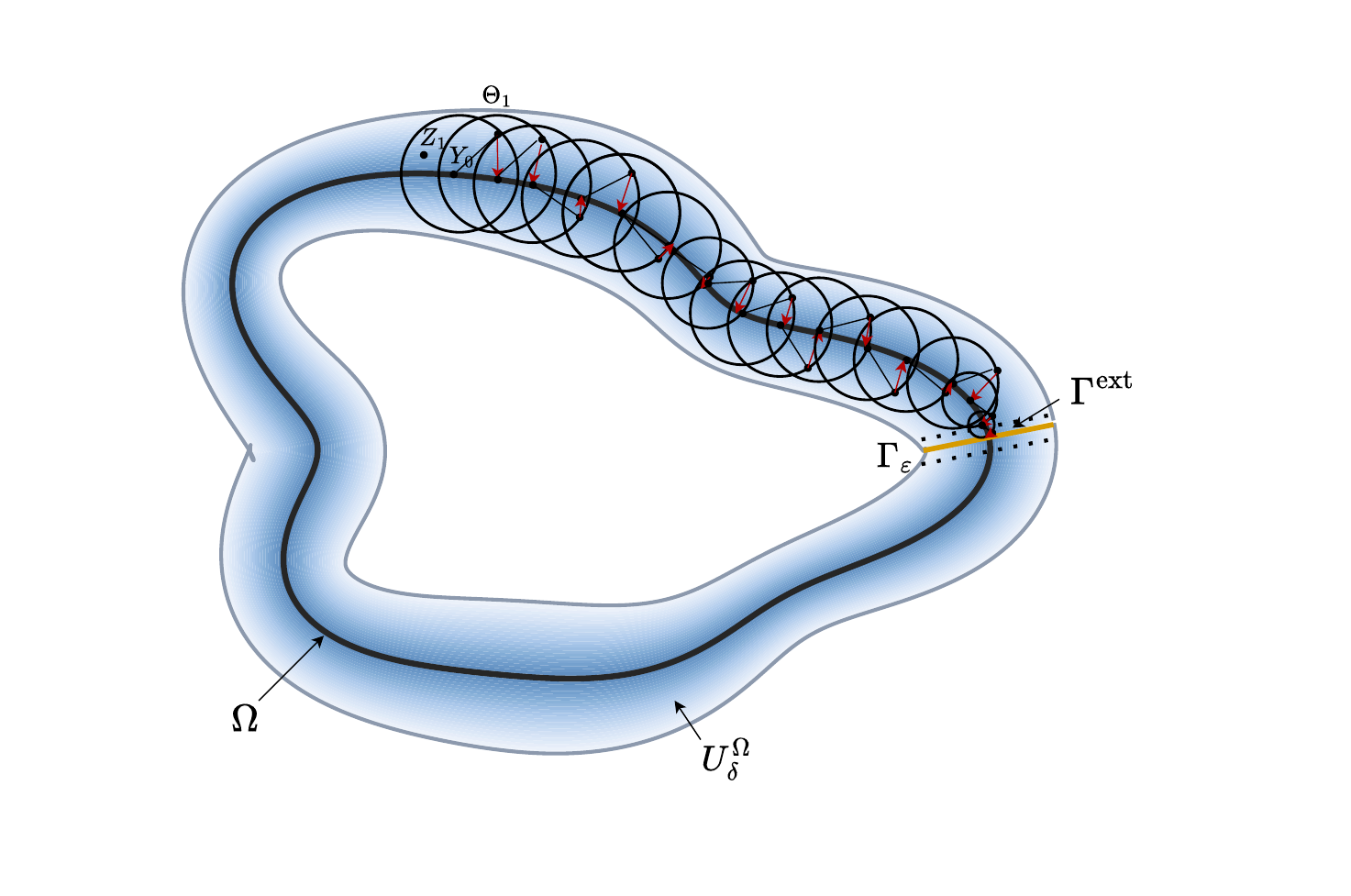}};

% Inset figure: move slightly inward to the left
\node[anchor=south west,inner sep=0] (inset) at (5.55,1.15)
{\includegraphics[width=1.8cm,trim=1.3cm 0.2cm 0cm 0cm,clip]{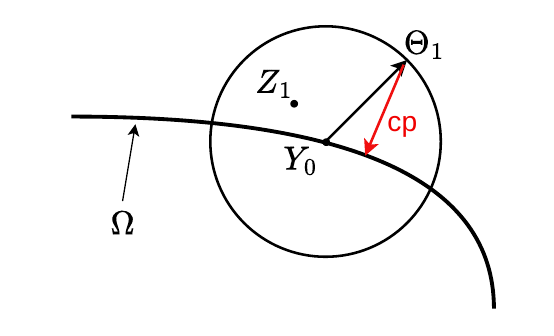}};

% Inset frame
\draw[black,thick] (5.55,1.15) rectangle (7.25,2.5);

% Zoom box on the main figure: adjusted to cover the local ball region
\draw[black,thick] (1.85,2.95) rectangle (2.65,3.95);

% Connecting lines
\draw[black,thin] (2.65, 2.95) -- (5.55, 1.15);
\draw[black,thin] (2.65, 3.95) -- (5.55, 2.5);

\end{tikzpicture}%
}
\caption{Schematic illustration of the MPWoS method on a one-dimensional manifold (curve) $\Omega$. The local Euclidean balls are chosen inside the tubular neighborhood $U_\delta^\Omega$, where the closest-point projection is uniquely defined. Sampled boundary points are projected back to $\Omega$, and the chain is stopped upon entering the boundary layer $\Gamma_\varepsilon$. The inset illustrates the local projection step: starting from a manifold point $Y_0\in\Omega$, a point $\Theta_1$ is sampled in the local Euclidean ball centered at $Y_0$ and then projected back to $\Omega$ by the closest-point map. }
%The local Euclidean balls are chosen inside the tubular neighborhood $U_\delta^\Omega$, where the closest-point projection is uniquely defined; the sampled boundary points are projected back to $\Omega$, and the chain is stopped when it enters the boundary layer $\Gamma_\varepsilon$ near the extended cut boundary $\Gamma^{\mathrm{ext}}$. The inset shows the local projection mechanism: starting from a manifold point $Y_0\in\Omega$, a point $\Theta_1$ is sampled in the local Euclidean ball centered at $Y_0$ and then projected back to $\Omega$ by the closest-point map.}
\label{fig:path}
\end{figure}

The representation formulas in Theorem~\ref{thm:final_representation} naturally lead to a Monte Carlo approximation based on independent sample paths. At each recursion step, we need to compute: (i). the admissible radius $r(Y_k)$; (ii). the projected state $Y_k$ and the interior sample $Z_{k+1}$; (iii). the compensation approximation $m_\ast(Z_{k+1})$. 

 Note that, in the boundary case, the theoretical representation is based on the first hitting step of the boundary. In practice, however, the admissible radius decreases as the chain approaches the boundary, which increases the recursion cost. For a prescribed $\varepsilon>0$, we therefore introduce %the boundary layer
\begin{equation}\label{eq:boundary_layer}
\Gamma_{\varepsilon}:=\big\{y\in\Omega:\delta_\Gamma(y)<\varepsilon\big\},
\end{equation}
where $\delta_\Gamma(y)$ denotes the effective distance to the extended boundary. The recursion is terminated once the chain enters $\Gamma_\varepsilon$, and the terminal value is approximated by the boundary datum at the nearest boundary point. In the closed-manifold case, the recursion is terminated when $W_k<\tau_W$, where $\tau_W>0$ is a prescribed tolerance.
%\begin{list}{(\roman{enumi})}{%
%\usecounter{enumi}
%\setlength{\leftmargin}{3.0em}
%\setlength{\labelwidth}{2.0em}
%\setlength{\labelsep}{0.8em}
%\setlength{\itemsep}{2pt}
%\setlength{\topsep}{4pt}
%}
%\item the admissible radius $r(Y_k)$ (Section~\ref{sec:radius_selection});
%%\item {\color{blue}the weight $\omega(Y_k)$ (Equation \eqref{eq:omega_explicit_formula})};
%\item the projected state $Y_k$ and the interior sample $Z_{k+1}$ (Section~\ref{sec:YZ});
%\item the compensation approximation $m_h(Z_{k+1})$ (Section~\ref{sec:M}).
%\end{list}

%In the boundary case, the theoretical representation is based on the first hitting step of the boundary. In practical computation, however, the admissible radius becomes increasingly small as the chain approaches the boundary, which substantially increases the recursion cost. Therefore, for a prescribed $\varepsilon>0$, we introduce the boundary layer
%\begin{equation}\label{eq:boundary_layer}
%\Gamma_{\varepsilon}:=\big\{y\in\Omega:\delta_\Gamma(y)<\varepsilon\big\},
%\end{equation}
%where $\delta_\Gamma(y)$ denotes the effective distance to the extended boundary. The projected chain is terminated once it enters $\Gamma_\varepsilon$, and the terminal value is approximated by the boundary datum at a nearest boundary point. %The parameter $\varepsilon$ controls both the terminal approximation error and the recursion cost near the boundary. 
%In the closed-manifold case, the recursion is terminated once $W_k<\tau_W$, where $\tau_W>0$ is a prescribed tolerance. 

%We next describe the quantities in (i)–(iii) and their practical construction.

\subsubsection{Adaptive choice of the radius}\label{sec:radius_selection}

The first quantity is the local radius $r(Y_k)$, which determines the Euclidean ball in the local representation formula and thus affects both efficiency and accuracy. The radius should be small enough to keep the ball inside the tubular neighborhood and, in the boundary case, away from the extended boundary, yet large enough to avoid excessive recursion. These requirements are closely related to the reach of the manifold, which reflects both local curvature and bottleneck geometry.

In principle, the radius may also be defined through quantities associated with the medial axis or the local feature size (see, e.g., \cite{AmentaBernEppstein1998,DeyGiesen2003,Giesen2006,Sugimoto2024}). However, such constructions are often difficult to evaluate efficiently and robustly in high-dimensional or complicated geometries. We therefore adopt a radius rule based on local geometric and boundary information. For this purpose, we define the local radius by
\begin{equation}\label{eq:radius_rule_unified}
r(y):=\theta\,\min\bigl\{\delta_{\Gamma}(y),\;\rho_{\mathrm{geo}}(y)\bigr\},\quad 0<\theta<1,
\end{equation}
where $\theta$ is a safety factor, $\delta_{\Gamma}(y)$ denotes an effective distance to the extended boundary, and $\rho_{\mathrm{geo}}(y)$ is an additional geometric scale used to control the local sampling radius. Moreover, the radius is required to satisfy $\mathbb{B}_{r(y)}^{n}(y)\subset U_\delta^\Omega$. We next specify the construction of $\delta_{\Gamma}(y)$ and $\rho_{\mathrm{geo}}(y)$.

{\bf (1) Computation of the effective distance $\delta_{\Gamma}(y)$.}
In the closed manifold case, we simply set $\delta_{\Gamma}(y):=+\infty$. Assume now that $\Omega$ has boundary $\Gamma$, and let $y_{\Gamma}:=\operatorname{cp}_{\Gamma}(y)\in\Gamma$ be a nearest boundary point to $y$. Denote by
\[
P_T(y_{\Gamma}):\mathbb R^n\to T_{y_{\Gamma}}\mathcal M,
\quad
P_N(y_{\Gamma}):\mathbb R^n\to N_{y_{\Gamma}}\mathcal M
\]
the orthogonal projections onto the tangent and normal spaces of $\mathcal M$ at $y_{\Gamma}$, respectively. We further define
\[
r_{\Gamma}:=y_{\Gamma}-y,
\quad
r_{\Gamma}^T:=P_T(y_{\Gamma})\,r_{\Gamma},
\quad
r_{\Gamma}^N:=P_N(y_{\Gamma})\,r_{\Gamma}.
\]
To construct a practical radius rule, we approximate the local geometry of the extended boundary near $y_{\Gamma}$ by the normal disk
\begin{equation}\label{eq:extended_boundary_disk}
E_{\Gamma}(y_{\Gamma})
:=
\bigl\{
y_{\Gamma}+\xi:\ \xi\in N_{y_{\Gamma}}\mathcal M,\ \|\xi\|_2\le \ell(y_{\Gamma})
\bigr\},
\end{equation}
where $\ell(y_{\Gamma})>0$ is a prescribed extension length, taken sufficiently small so that the associated local extension stays inside the tubular neighborhood where the closest-point projection is uniquely defined. We then define the effective distance to the extended boundary by
\begin{equation}\label{eq:delta_Gamma_general}
\delta_{\Gamma}(y)
:=
\operatorname{dist}\bigl(y,E_{\Gamma}(y_{\Gamma})\bigr)
=
\left(
\|r_{\Gamma}^T\|_2^2+\bigl(\|r_{\Gamma}^N\|_2-\ell(y_{\Gamma})\bigr)_+^2
\right)^{1/2},
\end{equation}
where $(a)_+:=\max\{a,0\}$. Thus, $\delta_{\Gamma}(y)$ is the distance from $y$ to the normal-disk approximation of the extended boundary at $y_{\Gamma}$. In the practical implementation, we take the extension length to be comparable to the local geometric scale and, for simplicity, choose $\ell(y_{\Gamma})=\rho_{\mathrm{geo}}(y_{\Gamma})$.

\begin{remark}
\itshape
The above construction is consistent with the 1D and 2D cases in {\rm\cite{Sugimoto2024}}. In $\dim(\mathcal M)=2$ and codimension-one case, the normal disk \eqref{eq:extended_boundary_disk} reduces to a line segment.  
%For point-cloud manifolds, the quantities
%$
%y_\Gamma$, $P_T(y_\Gamma)$, and $P_N(y_\Gamma)
%$
%in \eqref{eq:delta_Gamma_general} are approximated through local reconstruction, for instance by a local tangent-plane fit or a local polynomial patch together with a reconstructed boundary set. The effective distance $\delta_\Gamma(y)$ is then evaluated from the reconstructed tangent-normal decomposition through the same formula \eqref{eq:delta_Gamma_general}.
%Let $n_{\Gamma}$ be a unit normal vector at $y_{\Gamma}$. Then \eqref{eq:delta_Gamma_general} becomes
%\begin{equation}\label{eq:delta_Gamma_codim1}
%\delta_{\Gamma}(y)
%=
%\bigl\|r_{\Gamma}-\operatorname{clamp}\bigl(r_{\Gamma}\!\cdot n_{\Gamma},-\ell(y_{\Gamma}),\ell(y_{\Gamma})\bigr)\,n_{\Gamma}\bigr\|_2,
%\end{equation}
%where
%$\operatorname{clamp}(a,b,c):=\min\{\max\{a,b\},c\}.$
%When $\dim(\mathcal M)=1$, let
%$q_{\Gamma}:=r_{\Gamma}-(r_{\Gamma}\!\cdot t_{\Gamma})\,t_{\Gamma},$
%where $t_{\Gamma}$ is a unit tangent vector at $y_{\Gamma}$. Then
%\begin{equation}\label{eq:delta_Gamma_curve}
%\delta_{\Gamma}(y)=\begin{cases}
%|r_{\Gamma}\!\cdot t_{\Gamma}|, & \|q_{\Gamma}\|_2<\ell(y_{\Gamma}),\\
%\bigl\|r_{\Gamma}-\ell(y_{\Gamma})\,q_{\Gamma}/\|q_{\Gamma}\|_2\bigr\|_2,
%& \text{otherwise}.
%\end{cases}
%\end{equation}
\end{remark}

{\bf (2) Computation of $\rho_{\mathrm{geo}}(y)$.}
The quantity $\rho_{\mathrm{geo}}(y)$ reflects the local geometric scale near $y$ and reduces the radius in regions with complicated geometry, even when the point is away from the boundary. Thus, $\rho_{\mathrm{geo}}(y)$ complements the boundary term $\delta_{\Gamma}(y)$ in the radius rule.

For manifolds with explicitly available differential information, we take
\begin{equation}\label{eq:rho_geo_curvature}
\rho_{\mathrm{geo}}(y):=\frac{c_{\kappa}}{1+\kappa(y)},
\end{equation}
where $\kappa(y)$ is a curvature indicator and $c_{\kappa}>0$ is a fixed constant. This choice decreases the local radius in highly curved regions, while allowing a relatively large radius in regions of mild geometry.

For manifolds represented only through discrete samples, where stable curvature information is generally unavailable, we instead set
\begin{equation}\label{eq:rho_geo_spacing}
\rho_{\mathrm{geo}}(y):=c_h\,h(y),
\end{equation}
where $h(y)$ denotes a local neighborhood size and $c_h>0$ is a fixed constant. A typical choice of $h(y)$ is the average distance from $y$ to a prescribed number of nearest neighbors.

\subsubsection{Sampling variables and closest-point projection}\label{sec:YZ}
Once the local radius $r(Y_k)$ has been determined, the corresponding Euclidean ball $\mathbb{B}_{r(Y_k)}^{n}(Y_k)\subset\mathbb R^n$ is fixed.
The projected recursion then proceeds by sampling a boundary point
\begin{equation}\label{eq:projected_chain_update}
Y_{k+1}
=
\operatorname{cp}(\Theta_{k+1}),\quad \Theta_{k+1}
=
Y_k+r(Y_k)\zeta_{k+1},
\end{equation}
where $\zeta_{k+1}$ is uniformly distributed on the unit sphere $\mathbb S^{n-1}$ and generated by normalizing a standard Gaussian vector in $\mathbb R^n$,
%
% {\color{blue} The projected recursion then proceeds by sampling a point on the boundary of this ball. More precisely, we write
%\begin{equation*}
%\Theta_{k+1}
%=
%Y_k+r(Y_k)\zeta_{k+1},
%\qquad
%Y_{k+1}
%=
%\operatorname{cp}(\Theta_{k+1}),
%\end{equation*}
%where the uniform direction $\zeta_{k+1}$ on the sphere is generated by normalizing a standard Gaussian vector in $\mathbb R^n$.} 
%In practice, such a direction is generated by normalizing a standard Gaussian vector in $\mathbb R^n$. 
and $\operatorname{cp}$ is defined in \eqref{eq:cp_map}. For parameterized manifolds, the projection is available explicitly, while for implicit and point-cloud manifolds it is computed through a local Newton solve and through nearest-neighbor search with local reconstruction, respectively. Thus, the projected chain evolves on $\mathcal M$, while sampling is performed in the ambient ball.

We next consider the interior sample $Z_{k+1}$. 
% Given the normalized Green density $\rho$ in \eqref{eq:omega_explicit_formula}, the interior point is sampled according to
% \begin{equation}\label{eq:Z_sampling_density}
% Z_{k+1}\sim \rho(Y_k,\cdot)
% \quad \text{in } \;\mathbb{B}_{r(Y_k)}^{n}(Y_k).
% \end{equation}
 Using \eqref{eq:omega_explicit_formula}, the normalized Green density is given by $\rho(y,z)=\mathcal G(y,z)/\omega(y)=\sigma\mathcal G(y,z)/(1-c_{\sigma,r(y)})$ for $z\in\mathbb B_{r(y)}^n(y)$. The interior point is sampled according to
\begin{equation}\label{eq:Z_sampling_density}
Z_{k+1}\sim \rho(Y_k,\cdot)
\quad \text{in } \mathbb{B}_{r(Y_k)}^{n}(Y_k).
\end{equation}
This sample contributes to the local volume term associated with the source and compensation terms in the Green representation.

\subsubsection{Approximation of the compensation term}\label{sec:M}
%It remains to specify how the compensation term is approximated in practice. 
Although $m$ is defined in \eqref{eq:mismatch_definition}, it is generally unavailable in closed form since it depends on the ambient Laplacian of the closest-point extension of the unknown solution $u_\Omega$. Thus, evaluating $m$ exactly requires a priori information on the solution and its derivatives. Except for problems with known exact solutions, we replace $m$ by a computable approximation $m_\ast$ constructed from a coarse numerical predictor.

% For clarity of presentation, we describe the construction in the codimension-one setting, that is, for a hypersurface $\mathcal M\subset\mathbb R^n$ with $\dim\mathcal M=n-1$. In this case, after choosing a local unit normal field $n$ on $\mathcal M$, each point $x\in U_\delta^\Omega$ can be written as
 For clarity of presentation, we describe the construction for a codimension-one manifold $\mathcal M\subset\mathbb R^n$. Let $n$ be a local unit normal field on $\mathcal M$, and for $x\in U_\delta^\Omega$, write 
\begin{equation}\label{eq:normal_coordinate_codim_one}
x
=
y+\eta n(y),
\quad
y=\operatorname{cp}(x)\in\Omega\subset\mathcal M,
\end{equation}
where $\eta$ is the signed normal distance from $x$ to $\mathcal M$. 
% Since $m$ vanishes on $\Omega$, we approximate its values in $U_\delta^\Omega$ by a Taylor expansion along the normal direction. 
% {Since m vanishes on $\Omega$, its off-surface values can be locally
% approximated by a Taylor expansion in the normal direction.}
% Thanks to \eqref{eq:normal_coordinate_codim_one}, we write
 Since $m$ vanishes on $\Omega$, we use the normal Taylor expansion 
\begin{equation*}\label{eq:m_taylor_normal}
\begin{aligned}
m(x)=m(y+\eta n(y))
=
m(y)
+
\eta\,m_1(y)
+
\frac12\eta^2\,m_2(y)
+
O(|\eta|^3)
=
\eta\,m_1(y)
+
\frac12\eta^2\,m_2(y)
+
O(|\eta|^3),
\end{aligned}
\end{equation*}
where
\begin{equation*}\label{eq:m1_m2_definition}
m_1(y)
=
\frac{\partial}{\partial \eta}
m(y+\eta n(y))
\Big|_{\eta=0},
\quad
m_2(y)
=
\frac{\partial^2}{\partial \eta^2}
m(y+\eta n(y))
\Big|_{\eta=0}.
\end{equation*}

% In computation, the exact coefficients $m_1$ and $m_2$ are replaced by numerical coefficients computed from an auxiliary approximation $u_\ast$. Let
% \begin{equation}\label{eq:auxiliary_residual_normal}
% R_\ast(s;y)
% :=
% -\Delta(Eu_\ast)(y+s n(y))
% +
% \sigma u_\ast(y)
% -
% f_\Omega(y),
% \quad |s|\ll1.
% \end{equation}
% Here we have used the fact that $\operatorname{cp}(y+s n(y))=y,$ so that
% \[
% Eu_\ast(y+s n(y))=u_\ast(y),
% \quad
% Ef_\Omega(y+s n(y))=f_\Omega(y).
% \]
% If $u_\ast$ were replaced by the exact solution $u_\Omega$, then $R_\ast(s;y)$ would coincide with $m(y+s n(y))$.
In computation, the coefficients in the normal expansion are approximated from an auxiliary predictor $u_\ast$. For this purpose, we define
%\begin{equation}\label{eq:auxiliary_residual_normal}
$R_\ast(s;y)=-\Delta(Eu_\ast)(y+s n(y))+\sigma u_\ast(y)-f_\Omega(y),$ $|s|\ll1.$
%\end{equation}
Since $\operatorname{cp}(y+s n(y))=y$, we have $Eu_\ast(y+s n(y))=u_\ast(y)$ and $Ef_\Omega(y+s n(y))=f_\Omega(y)$.
If $u_\ast=u_\Omega$, then $R_\ast(s;y)$ coincides with the compensation term along the normal direction. 
The compensation term is approximated from a coarse predictor $u_\ast$ through a second-order local normal expansion
\begin{equation*}
m_\ast(x)
=
\eta(x)\,m_{1,\ast}(\operatorname{cp}(x))
+
\frac12\eta(x)^2\,m_{2,\ast}(\operatorname{cp}(x)),
\quad x\in U_\delta^\Omega.
\end{equation*}
where\begin{equation*}
m_{1,\ast}(y)
=
\frac{
R_\ast(h_\eta;y)-R_\ast(-h_\eta;y)
}{
2h_\eta
},
\quad
m_{2,\ast}(y)
=
\frac{
R_\ast(h_\eta;y)-2R_\ast(0;y)+R_\ast(-h_\eta;y)
}{
h_\eta^2
}.
\end{equation*}
 %The resulting approximation, denoted by $m_\ast$, is used in place of the unavailable quantity $m$ in the path simulation. 
 Consequently, the source term in \eqref{eq:boundary_representation} is replaced by $f(Z_{k+1})+m_\ast(Z_{k+1}).$
%In computation, {\color{blue}$m_1$ and $m_2$ are approximated from $u_\ast$. Define}
%\begin{equation}\label{eq:auxiliary_residual_normal}
%R_\ast(s;y)
%:=
%-\Delta(Eu_\ast)(y+s n(y))
%+
%\sigma u_\ast(y)
%-
%f_\Omega(y),
%\qquad |s|\ll1.
%\end{equation}
%{\color{blue}Since $\operatorname{cp}(y+s n(y))=y$, we have $Eu_\ast(y+s n(y))=u_\ast(y)$ and $Ef_\Omega(y+s n(y))=f_\Omega(y)$. If $u_\ast=u_\Omega$, then $R_\ast(s;y)$ coincides with $m(y+s n(y))$.}
%{\color{blue}Let $h_\eta>0$ be a small normal finite-difference step. We approximate the normal derivatives in the Taylor expansion by
%\begin{equation*}
%m_{1,\ast}(y)
%=
%\frac{
%R_\ast(h_\eta;y)-R_\ast(-h_\eta;y)
%}{
%2h_\eta
%},
%\qquad
%m_{2,\ast}(y)
%=
%\frac{
%R_\ast(h_\eta;y)-2R_\ast(0;y)+R_\ast(-h_\eta;y)
%}{
%h_\eta^2
%}.
%\end{equation*}
%The resulting second-order approximation of the compensation term is given by
%\begin{equation*}
%m_\ast(x)
%=
%\eta(x)\,m_{1,\ast}(\operatorname{cp}(x))
%+
%\frac12\eta(x)^2\,m_{2,\ast}(\operatorname{cp}(x)),
%\qquad x\in U_\delta^\Omega.
%\end{equation*}
%Thus, in the path simulation, the unavailable value $m(Z_{k+1})$ is replaced by $m_\ast(Z_{k+1})$, and the source term in the local Green representation is replaced by
%$f(Z_{k+1})+m_\ast(Z_{k+1})$. }
A simpler baseline variant is obtained by setting $m_\ast\equiv0$, in which case no auxiliary predictor or derivative reconstruction is required.
 The above construction is presented for the codimension-one case. For codimension $q>1$, the scalar normal distance is replaced by local normal coordinates, and the compensation approximation follows from the corresponding multivariate Taylor expansion.

\subsubsection{The Monte Carlo method}

The representation formulas in Theorem \ref{thm:final_representation} lead directly to a Monte Carlo approximation. The expectations are estimated by independent simulated paths of the projected chain $\{Y_k, Z_k\}_{k\ge0}$. We denote the number of independent paths by $N$, and use the superscript $i$ to indicate the quantities generated in the $i$-th path. At each step of one path, we may use more than one interior sample in order to reduce the variance of the local volume contribution. More precisely, for the $i$-th path and the $k$-th time step, we define
\begin{equation}\label{eq:local_average_source}
F_k^i
:=
\frac{1}{N_V}
\sum_{\ell=1}^{N_V}
\left[
f(Z_{k+1}^{i,\ell})
+
m_\ast(Z_{k+1}^{i,\ell})
\right],
\quad
Z_{k+1}^{i,\ell}
\overset{\mathrm{i.i.d.}}{\sim}
\rho(Y_k^i,\cdot)
\ \text{in}\
B_{r(Y_k^i)}^n(Y_k^i).
\end{equation}
%More precisely, for the $i$-th path and the $k$-th step, we draw
%\[
%Z_{k+1}^{i,\ell}
%\sim
%\rho(Y_k^i,\cdot),
%\quad
%\ell=1,\dots,N_V,
%\]
%independently in the ball $B_{r(Y_k^i)}^n(Y_k^i)$, and define the local averaged source contribution by
%\begin{equation}\label{eq:local_average_source}
%F_k^i
%:=
%\frac{1}{N_V}
%\sum_{\ell=1}^{N_V}
%\left[
%f(Z_{k+1}^{i,\ell})
%+
%m_\ast(Z_{k+1}^{i,\ell})
%\right].
%\end{equation}
%When $N_V=1$, this reduces to the single-sample form appearing in Theorem \ref{thm:final_representation}.

%Then we can construct a Monte Carlo procedure based on the random sample associated with each path. 
Then, the contribution of the $i$-th path is defined by
\begin{equation}\label{eq:path_score_definition}
S^i(y_0)
:=
\begin{cases}
\displaystyle
W_{K^i}^i
g_\Omega\!\left(\operatorname{cp}_\Gamma(Y_{K^i}^i)\right)
+
\sum_{k=0}^{K^i-1}
W_k^i\,\omega(Y_k^i)\,F_k^i,
& \text{if } \Omega \text{ has boundary } \Gamma,
\\[6pt]
\displaystyle
\sum_{k=0}^{K^i-1}
W_k^i\,\omega(Y_k^i)\,F_k^i,
& \text{if } \Omega \text{ is closed},
\end{cases}
\end{equation}
where the stopping step $K^i$ is defined by
\[
K^i=
\begin{cases}
\inf\{k\ge 0:\,Y_k^i\in\Gamma_\varepsilon\}, & \text{if $\Omega$ has boundary},\\[4pt]
\inf\{k\ge 0:\,W_k^i<\tau_W\}, & \text{if $\Omega$ is closed}.
\end{cases}
\]

%Let $y_0\in\Omega$ be an interior point, and let $\{Y_k^i\}_{k\ge0}$ and $\{W_k^i\}_{k\ge0}$, $i=1,\dots,N$, be $N$ independent realizations of the projected chain and the corresponding weights, respectively. Then
% \begin{equation}\label{eq:mc_estimator}
% u_\Omega(y_0)
% \approx
% \lim_{N \to \infty} \frac{1}{N} \sum_{i=1}^N S^i(y_0)=
% \begin{cases}
%     \mathbb{E}\Big[\;W_{K^i}^i
%     g_\Omega\!\left(\operatorname{cp}_\Gamma(Y_{K^i}^i)\right)
%     +
%     \sum\limits_{k=0}^{K^i-1}
%     W_k^i\,\omega(Y_k^i)\,F_k^i\;\Big],
%     & \text{if } \Omega \text{ has boundary } \Gamma,
%     \\[10pt]
%     \displaystyle
%     \mathbb{E}\Big[\;\sum_{k=0}^{K^i-1}
%     W_k^i\,\omega(Y_k^i)\,F_k^i\;\Big],
%     & \text{if } \Omega \text{ is closed}.
% \end{cases}
% \end{equation}
Based on the above construction, the solution of \eqref{eq:boundary_representation} and \eqref{eq:boundaryless_representation} can be computed by
 \begin{equation}\label{eq:mc_estimator}
u_\Omega(y_0)
\approx
\lim_{N \to \infty} \frac{1}{N} \sum_{i=1}^N S^i(y_0)=
\mathbb{E}\bigl[S^i(y_0)\bigr].
\end{equation}

Finally, the above derivations lead to the Monte Carlo algorithm outlined in Algorithm \ref{alg:projected_wos}.

\begin{algorithm}[ht!]
\caption{Projected Walk on Spheres method for \eqref{eq:surface_pde}}
\label{alg:projected_wos}
\begin{algorithmic}[1]
\Require The number of paths $N$ and the number of samples per path $N_V$;
\Require The boundary threshold $\varepsilon$ and the truncation threshold $\tau_W$;
\Require The initial point $y_0\in\Omega$.
% \Ensure The Monte Carlo estimator 
% $\bar u_{\Omega,N}(y_0)=N^{-1}\sum_{i=1}^{N}S^i(y_0)$.

\For{$i=1,2,\ldots,N$}
    \State Set $k=0$, $Y_0^i=y_0$, and $W_0^i=1$.
    \While{the stopping criterion is not satisfied}
        \State Compute the radius $r(Y_k^i)$ by \eqref{eq:radius_rule_unified}.
        \State Construct the random variable $Y_{k+1}^i$ by the closest-point projection \eqref{eq:projected_chain_update}.
        \State Draw $Z_{k+1}^{i,\ell}$, $\ell=1,\ldots,N_V$, according to the density $\rho(Y_k^i,\cdot)$.
        \State Update the weight by $W_{k+1}^i=W_k^i\,c_{\sigma,r(Y_k^i)}$.
        \If{$\Omega$ has boundary $\Gamma$ and $Y_{k+1}^i\in\Gamma_\varepsilon$}
            \State Set $K^i=k+1$ and compute $S^i(y_0)$ by \eqref{eq:path_score_definition}.
        \ElsIf{$\Omega$ is closed and $W_{k+1}^i<\tau_W$}
            \State Set $K^i=k+1$ and compute $S^i(y_0)$ by \eqref{eq:path_score_definition}.
        \Else
            \State Set $k=k+1$.
        \EndIf
    \EndWhile
\EndFor
{\Statex \hspace{-\algorithmicindent}\textbf{Output:}\, The Monte Carlo estimator 
$\bar u_{\Omega,N}(y_0)=N^{-1}\sum_{i=1}^{N}S^i(y_0)$.}
\end{algorithmic}
\end{algorithm}

%Compared with existing projected Walk-on-Spheres methods for surface PDEs (cf. \cite{Sugimoto2024}), the proposed approach incorporates several significant extensions. First, we establish a local probabilistic representation and the associated projected recursion for the screened Poisson equation on embedded manifolds. Second, the framework introduces a compensation mechanism to mitigate the discrepancy between the ambient Laplacian and the intrinsic operator, coupled with an adaptive radius strategy informed by local geometric properties. Third, the method is designed for a broad class of embedded manifolds, encompassing parametrized manifolds, implicitly defined manifolds, and point-cloud representations, and it naturally extends to high-dimensional settings. We next proceed with a rigorous error analysis of the proposed projected Walk-on-Spheres method.

\section{The error estimate and theoretical analysis}\label{sec:error_analysis}

  In this section, we derive error estimates for the projected Monte Carlo estimator introduced in Section \ref{sec:algorithm}. We first record a simple sufficient condition for the uniform upper bound on the screened weights.

% \begin{lemma}\label{lem:uniform_screened_weight}
% Assume that $\Omega$ is compact, and that the radius function has a positive lower bound $r_{\min}>0$ on $\Omega$. Then the screened factor $c_{\sigma,r(y)}$ satisfies
% \begin{equation}\label{eq:uniform_screened_weight}
% c_{\sigma,r(y)}\le c_\ast<1,
% \qquad
% c_\ast=
% \left(
% 1+\frac{\sigma r_{\min}^2}{2n}
% \right)^{-1}.
% \end{equation}
% \end{lemma}
 
\begin{lemma}\label{lem:uniform_screened_weight}
Assume that $\Omega$ is compact. For a fixed $\varepsilon>0$, we define
\begin{equation}\label{omegaepsilon}
\Omega_\varepsilon
:=
\begin{cases}
\Omega, & \text{if }\; \Gamma=\emptyset,\\
\{y\in\Omega:\delta_\Gamma(y)\ge\varepsilon\}, & \text{if }\; \Gamma\ne\emptyset.
\end{cases}
\end{equation}
Assume that the radius function has a positive lower bound on $\Omega_\varepsilon$. Under the radius rule \eqref{eq:radius_rule_unified}, one may take $r_{\min}:=\theta\min\big\{\varepsilon,\inf_{y\in\Omega_\varepsilon}\rho_{\rm geo}(y)\big\}>0$. Then the screened factor $c_{\sigma,r(y)}$ satisfies
\begin{equation}\label{eq:uniform_screened_weight}
c_{\sigma,r(y)}\le c_\ast:=
\Big(1+\frac{\sigma r_{\min}^2}{2n}\Big)^{-1}<1,
%\quad c_\ast=\Big(1+\frac{\sigma r_{\min}^2}{2n}\Big)^{-1},
\quad y\in\Omega_\varepsilon .
\end{equation}
\end{lemma}

\begin{proof}
By \eqref{eq:def_c_sigma_r} and the definition of the modified Bessel function of the first kind $I_\nu$, we obtain
\begin{equation*}
c_{\sigma,r}
=
\frac{(\sqrt{\sigma}\,r)^\nu}
{2^\nu\Gamma(\nu+1)I_\nu(\sqrt{\sigma}\,r)}
=\bigg[
\sum_{j=0}^{\infty}
\frac{(\sigma r^2/4)^j}{j!(\nu+1)_j}
\bigg]^{-1},\quad \nu=\frac{n-2}{2}.
\end{equation*}
where $(\nu+1)_j=(\nu+1)(\nu+2)\cdots(\nu+j)$ for $j\ge1$ and $(\nu+1)_0=1$. Since every term in the series is nonnegative, we conclude that
\begin{equation*}
\sum_{j=0}^{\infty}
\frac{(\sigma r^2/4)^j}{j!(\nu+1)_j}
\ge
1+\frac{\sigma r^2/4}{\nu+1}
=
1+\frac{\sigma r^2}{4(\nu+1)}=1+\frac{\sigma r^2}{2n}.
\end{equation*}
% Using the positive lower bound $r(y)\ge r_{\min}$ on $\Omega$, it follows that, for all $y\in\Omega$,
% \begin{equation*}
% c_{\sigma,r(y)}
% \le
% \bigg(
% 1+\frac{\sigma r_{\min}^2}{2n}
% \bigg)^{-1}
% =:c_\ast<1,
% \end{equation*}
% where $c_\ast<1$ follows from $\sigma>0$ and $r_{\min}>0$. This proves \eqref{eq:uniform_screened_weight}.
Using the positive lower bound $r(y)\ge r_{\min}$ on $\Omega_\varepsilon$, it follows that, for all $y\in\Omega_\varepsilon$,
\begin{equation*}
c_{\sigma,r(y)}
\le
\Big(1+\frac{\sigma r_{\min}^2}{2n}\Big)^{-1}
=c_\ast<1,
\end{equation*}
where $c_\ast<1$ follows from $\sigma>0$ and $r_{\min}>0$.

It remains to justify the stated choice of $r_{\min}$ under the radius rule \eqref{eq:radius_rule_unified}. In the manifold-with-boundary case, the projected chain is stopped once it enters the boundary layer $\Gamma_\varepsilon$, and hence all radii used before stopping are evaluated at points in $\Omega_\varepsilon$. For any $y\in\Omega_\varepsilon$, one has $\delta_\Gamma(y)\ge\varepsilon$. Therefore, by \eqref{eq:radius_rule_unified}, we have 
$
r(y)
\ge
\theta
\min\big\{
\varepsilon,\inf_{y\in\Omega_\varepsilon}\rho_{\rm geo}(y)
\big\}
=
r_{\min}.
$
This proves \eqref{eq:uniform_screened_weight}.
\end{proof}

The preceding lemma provides a simple sufficient condition for the uniform bound $c_{\sigma,r(y)}\le c_\ast<1$. Then, we use this bound to control the accumulated weights along the projected chain.

\begin{lemma}\label{lem:variance_bound}
Let $\bar{u}_{\Omega,N}(y_0):=\frac{1}{N}\sum_{i=1}^N S^i(y_0)$ denote the numerical solution associated with the exact closest-point map $\operatorname{cp}$, with $S^i(y_0)$ given in \eqref{eq:path_score_definition}.  
Assume that $\Omega$ is compact and that the radius function admits a positive lower bound $r_{\min}>0$ on $\Omega_\varepsilon$, where $\Omega_\varepsilon$ is defined in \eqref{omegaepsilon}. Suppose further that there exist constants $M_f,M_g,M_m>0$ {such that $|f(x)|\le M_f$ for all $x\in U_\delta^\Omega$}, $|g_\Omega(y)|\le M_g$ for all $y\in\Gamma$ when $\Omega$ has boundary $\Gamma$, and $|m_\ast(x)|\le M_m$ for all $x\in U_\delta^\Omega$. Then, we have
\begin{equation}\label{eq:variance_bound_first}
\operatorname{Var}\bigl(\bar{u}_{\Omega,N}(y_0)\bigr)
\le  \frac{C_1}{N}+\frac{C_2}{N N_V},
\end{equation}
where the positive constants $C_1$ and $C_2$ are independent of $N$ and $N_V$. % In the boundary case, the constant $C$ depend on $\varepsilon$ through $r_{\min}$ and $c_\ast$. %{\color{red}Consequently, $\operatorname{Var}\bigl(\bar{u}_{\Omega,N}(y_0)\bigr)\le\mathcal O(N^{-1})$.}
\end{lemma}

\begin{proof}
We first present the argument for the boundary case. Since the random samples
$\{S^i(y_0)\}_{i=1}^N$ are independent and identically distributed, we have
$\operatorname{Cov}(S^i(y_0),\allowbreak S^j(y_0))=0$, for $i\neq j$. Thus,
\begin{equation*}
\begin{split}
\operatorname{Var}\bigl(\bar{u}_{\Omega,N}(y_0)\bigr)
&=
\operatorname{Var}\!\bigg(
\frac{1}{N}\sum_{i=1}^N \bigg[W_{K^i}^i
g_\Omega\!\bigl(\operatorname{cp}_\Gamma(Y_{K^i}^i)\bigr)
+
\sum_{k=0}^{K^i-1} W_k^i\,\omega(Y_k^i)\,F_k^i\bigg]
\bigg)\\
&\le
\frac{2}{N}\operatorname{Var}\!\big[
W_K\,g_\Omega\!\bigl(\operatorname{cp}_\Gamma(Y_K)\bigr)
\big]
+
\frac{2}{N}\operatorname{Var}\!\bigg[\;
\sum_{k=0}^{K-1} W_k\,\omega(Y_k)\,F_k \;\bigg].
\end{split}
\end{equation*}
% If $\Omega$ is closed or $g_\Omega\equiv0$, then the terminal term vanishes identically. Therefore, in the remainder of the proof, we only consider the case where $\Omega$ has boundary $\Gamma$ and $g_\Omega\not\equiv0$.
{If $g_\Omega\equiv0$, the terminal term vanishes identically. The closed-manifold case is obtained by omitting the terminal boundary contribution. Therefore, in the remainder of the proof, it suffices to estimate the two terms on the right hand side in the boundary case with $g_\Omega\not\equiv0$.}

For the first term, since $0<W_K\le1$ and $|g_\Omega(y)|\le M_g$ for all $y\in\Gamma$, we have
\begin{equation*}
\begin{split}
\operatorname{Var}\!\big[
W_K\,g_\Omega\!\bigl(\operatorname{cp}_\Gamma(Y_K)\bigr)
\big]
&=\mathbb{E}\bigg[\Big(W_K\,g_\Omega\!\bigl(\operatorname{cp}_\Gamma(Y_K)\bigr)\Big)^2\bigg]
-
\Big(\mathbb{E}\big[W_K\,g_\Omega\!\bigl(\operatorname{cp}_\Gamma(Y_K)\bigr)\big]\Big)^2\\
&\le
\mathbb{E}\bigg[
\Big(W_K\,g_\Omega\!\bigl(\operatorname{cp}_\Gamma(Y_K)\bigr)\Big)^2
\bigg]
\le
M_g^2\,\mathbb{E}\big[(W_K)^2\big]\le M_g^2.
\end{split}
\end{equation*}

% For the second term, by \eqref{eq:local_average_source} and the boundedness of $f_\Omega$ and $m_\ast$, we have
% \begin{equation*}
% |F_k|
% \le
% \frac{1}{N_V}\sum_{\ell=1}^{N_V}
% \big(
% |f_\Omega(Z_{k+1}^{\ell})|+|m_\ast(Z_{k+1}^{\ell})|
% \big)
% \le M_f+M_m.
% \end{equation*}
{For the second term, by \eqref{eq:local_average_source} and the boundedness of $f$ and $m_\ast$, we have
\begin{equation*}
|F_k|
\le
\frac{1}{N_V}\sum_{\ell=1}^{N_V}
\big(
|f(Z_{k+1}^{\ell})|+|m_\ast(Z_{k+1}^{\ell})|
\big)
\le M_f+M_m.
\end{equation*}}
By Lemma \ref{lem:uniform_screened_weight}, there exists $c_\ast\in(0,1)$ such that $c_{\sigma,r(y)}\le c_\ast$ for all {$y\in\Omega_{\varepsilon}$}. Also, by \eqref{eq:omega_explicit_formula}, we have
\begin{equation}\label{omegaboundx}
\omega(y)=\frac{1-c_{\sigma,r(y)}}{\sigma}\le \frac{1}{\sigma},
\quad W_k=\prod_{j=0}^{k-1}c_{\sigma,r(Y_j)}\le c_\ast^k.
\end{equation}

Fixing the projected path, write $F_k=\Phi_k+R_k^{\mathrm{vol}}$, where $\Phi_k$ is the mean of the local Green-density estimator, and $R_k^{\mathrm{vol}}$ is the corresponding volume-sampling residual. Then $\mathbb{E}[R_k^{\mathrm{vol}}]=0$, and since $F_k$ is the average of $N_V$ independent Green-density samples bounded by $M_f+M_m$, we have
\begin{equation*}
\operatorname{Var}_{\mathrm{vol}}(R_k^{\mathrm{vol}})
=
\operatorname{Var}_{\mathrm{vol}}(F_k-\Phi_k)
=
\operatorname{Var}_{\mathrm{vol}}(F_k)
\le
\frac{(M_f+M_m)^2}{N_V}.
\end{equation*}
Moreover, the bound $|\Phi_k|\le M_f+M_m$ holds uniformly for all $k$. Combining this relation with the decomposition of $F_k$, we deduce that
\begin{equation}\label{term11}
\begin{split}
\operatorname{Var}\!\bigg[
\sum_{k=0}^{K-1} W_k\,\omega(Y_k)\,F_k
\bigg]
&=
\operatorname{Var}\!\bigg[
\sum_{k=0}^{K-1} W_k\,\omega(Y_k)\,\Phi_k
+
\sum_{k=0}^{K-1} W_k\,\omega(Y_k)\,R_k^{\mathrm{vol}}
\bigg]\\
&\le
2\operatorname{Var}\!\bigg[
\sum_{k=0}^{K-1} W_k\,\omega(Y_k)\,\Phi_k
\bigg]
+
2\operatorname{Var}\!\bigg[
\sum_{k=0}^{K-1} W_k\,\omega(Y_k)\,R_k^{\mathrm{vol}}
\bigg].
\end{split}
\end{equation}

For the first term of  \eqref{term11}, using $|\Phi_k|\le M_f+M_m$, we obtain
\begin{equation*}
\begin{split}
\operatorname{Var}\!\bigg[
\sum_{k=0}^{K-1} W_k\,\omega(Y_k)\,\Phi_k
\bigg]
&\le
\mathbb{E}\bigg[
\bigg(
\sum_{k=0}^{K-1} W_k\,\omega(Y_k)\,\Phi_k
\bigg)^2
\bigg]\\
&\le
\mathbb{E}\bigg[
\bigg(
\frac{M_f+M_m}{\sigma}\sum_{k=0}^{K-1}c_\ast^k
\bigg)^2
\bigg]
\le
\bigg(
\frac{M_f+M_m}{\sigma(1-c_\ast)}
\bigg)^2.
\end{split}
\end{equation*}

For the second term of \eqref{term11}, fixing the projected path and using the independence of the volume samples used at different steps, we have
\begin{eqnarray*}
&&\operatorname{Var}_{\mathrm{vol}}\!\bigg[
\sum_{k=0}^{K-1} W_k\,\omega(Y_k)\,R_k^{\mathrm{vol}}
\bigg]=\sum_{k=0}^{K-1}
\big(W_k\,\omega(Y_k)\big)^2
\operatorname{Var}_{\mathrm{vol}}(R_k^{\mathrm{vol}})\\
&&\hspace{120pt}\le
\frac{(M_f+M_m)^2}{N_V}
\sum_{k=0}^{K-1}
\big(W_k\,\omega(Y_k)\big)^2\\
&&\hspace{120pt}\le
\frac{(M_f+M_m)^2}{N_V\sigma^2}
\sum_{k=0}^{K-1}c_\ast^{2k}
\le
\frac{(M_f+M_m)^2}{N_V\sigma^2(1-c_\ast^2)}.
\end{eqnarray*}
Together with the estimate for the first part, this gives
\begin{equation*}
\operatorname{Var}\!\bigg[
\sum_{k=0}^{K-1} W_k\,\omega(Y_k)\,F_k
\bigg]
\le
2\bigg(
\frac{M_f+M_m}{\sigma(1-c_\ast)}
\bigg)^2
+
\frac{2(M_f+M_m)^2}{N_V\sigma^2(1-c_\ast^2)}.
\end{equation*}

Combining the above estimates yields
\begin{equation*}
\begin{split}
\operatorname{Var}\bigl(\bar{u}_{\Omega,N}(y_0)\bigr)
&\le
\frac{2M_g^2}{N}
+
\frac{4}{N}
\bigg(
\frac{M_f+M_m}{\sigma(1-c_\ast)}
\bigg)^2
+
\frac{4(M_f+M_m)^2}{N N_V\sigma^2(1-c_\ast^2)}\le
\frac{C_1}{N}
+
\frac{C_2}{N N_V},
\end{split}
\end{equation*}
where $C_1$ and $C_2$ are constants independent of $N$ and $N_V$. In the manifold-with-boundary case, these constants may depend on $\varepsilon$ through $r_{\min}$, equivalently through $c_\ast$. Since $N_V\ge1$, \eqref{eq:variance_bound_first} follows.  
%\begin{equation*}
%\operatorname{Var}\bigl(\bar{u}_{\Omega,N}(y_0)\bigr)
%\le
%\frac{C_1+C_2}{N}
%=
%\mathcal O(N^{-1}).
%\end{equation*}
%
In the closed manifold case, the terminal term is absent, and the same argument yields the same bound with $M_g=0$. This ends the proof.
\end{proof}

%Estimate \eqref{eq:variance_bound_first} distinguishes the variance generated by the projected chain from that generated by the Green-density volume sampling. The former contributes the term $C_1/N$, while the latter contributes $C_2/(N N_V)$. Thus, increasing $N_V$ reduces only the volume-sampling contribution to the variance, while the leading convergence rate with respect to the number of paths remains $N^{-1/2}$. 
We next impose the following boundary regularity condition.

%\begin{assumption}\label{ass:boundary_lipschitz}
%The following boundary regularity condition holds. When $\Omega$ has boundary $\Gamma$, 
%% let $\delta_\Gamma(y):=\operatorname{dist}(y,\Gamma)$ for $y$ near $\Gamma$. 
%{\color{blue}let $\delta_\Gamma(y)$ denote the effective boundary distance used in the stopping rule and in the radius construction.}
%We assume that $u_\Omega$ is Lipschitz continuous in a neighborhood of $\Gamma$, namely, there exists a constant $L_\Gamma>0$ such that
%\begin{equation}\label{eq:boundary_lipschitz}
%|u_\Omega(y)-u_\Omega(\operatorname{cp}_\Gamma(y))|
%\le L_\Gamma\,\delta_\Gamma(y),
%\quad y\in\Gamma_\varepsilon.
%\end{equation}
%\end{assumption}

%We next consider the error introduced by the boundary stopping rule.
\begin{assumption}\label{ass:boundary_lipschitz}
Assume the following boundary regularity condition holds. When $\Omega$ has boundary $\Gamma$, let $\delta_\Gamma(y)$ denote the effective boundary distance used in the stopping rule and in the radius construction. We assume that $u_\Omega$ is Lipschitz continuous in a neighborhood of $\Gamma$, i.e., there exists a constant $L_\Gamma>0$ such that
\begin{equation}\label{eq:boundary_lipschitz}
|u_\Omega(y)-u_\Omega(\operatorname{cp}_\Gamma(y))| \le L_\Gamma \,\delta_\Gamma(y), \quad y \in \Gamma_\varepsilon.
\end{equation}
\end{assumption}

We now estimate the error caused by the boundary stopping procedure under Assumption \ref{ass:boundary_lipschitz}.

%\begin{lemma}\label{lem:boundary_layer_error}
%Assume that $\Omega$ has boundary $\Gamma$. Under Assumption \ref{ass:boundary_lipschitz}, the error caused by the boundary stopping procedure satisfies
%\begin{equation}\label{eq:boundary_layer_error}
%\big|
%\mathbb E\!\big[
%W_K\,u_\Omega(Y_K)
%\big]
%-
%\mathbb E\!\big[
%W_K\,g_\Omega(\operatorname{cp}_\Gamma(Y_K))
%\big]
%\big|
%\le
%L_\Gamma\,\varepsilon,
%\end{equation}
%{\color{blue} where $L_\Gamma$ is the boundary Lipschitz constant in Assumption \ref{ass:boundary_lipschitz}.}
%\end{lemma}
\begin{lemma}\label{lem:boundary_layer_error}
Assume that $\Omega$ has boundary $\Gamma$. Under Assumption \ref{ass:boundary_lipschitz}, the error introduced by the boundary stopping procedure satisfies
\begin{equation}\label{eq:boundary_layer_error}
\Big|\mathbb E\big[ W_K\,u_\Omega(Y_K) \big]-\mathbb E\big[ W_K\,g_\Omega(\operatorname{cp}_\Gamma(Y_K)) \big]\Big|\le L_\Gamma\,\varepsilon,
\end{equation}
where $L_\Gamma$ is the boundary Lipschitz constant in \eqref{eq:boundary_lipschitz}.
\end{lemma}
\begin{proof}
Since $Y_K\in\Gamma_\varepsilon$ by the stopping rule, we have $\delta_\Gamma(Y_K)<\varepsilon$. Moreover, since $u_\Omega=g_\Omega$ on $\Gamma$, it follows that
$u_\Omega(\operatorname{cp}_\Gamma(Y_K))=g_\Omega(\operatorname{cp}_\Gamma(Y_K)).$
Consequently, the difference of the two expectations can be rewritten as
\begin{equation*}
\begin{split}
\Big|\mathbb E\!\big[W_K\,u_\Omega(Y_K)\big]-\mathbb E\!\big[W_K\,g_\Omega(\operatorname{cp}_\Gamma(Y_K))\big]\Big|
&=\Big|\mathbb E\!\big[W_K\big(
u_\Omega(Y_K)-u_\Omega(\operatorname{cp}_\Gamma(Y_K))
\big)
\big]
\Big|\\
&\le
\mathbb E\!\big[
W_K
\big|
u_\Omega(Y_K)-u_\Omega(\operatorname{cp}_\Gamma(Y_K))
\big|
\big].
\end{split}
\end{equation*}
Since $0<W_K\le 1$, the boundary Lipschitz estimate \eqref{eq:boundary_lipschitz} implies that
\begin{equation*}
\begin{split}
\Big|\mathbb E\!\big[
W_K\,u_\Omega(Y_K)
\big]-\mathbb E\!\big[
W_K\,g_\Omega(\operatorname{cp}_\Gamma(Y_K))
\big]
\Big|
&\le
\mathbb E\!\big[
\big|
u_\Omega(Y_K)-u_\Omega(\operatorname{cp}_\Gamma(Y_K))
\big|
\big]\le
L_\Gamma\,\mathbb E[\delta_\Gamma(Y_K)]
\le
L_\Gamma\,\varepsilon.
\end{split}
\end{equation*}
This ends the proof.
\end{proof}

%We now turn to the closed manifold case and estimate the truncation error caused by the stopping rule $W_K<\tau_W$.
%We now consider the closed manifold case and estimate the truncation error from the stopping rule. % $W_K<\tau_W$.

% \begin{lemma}\label{lem:truncation_error}
% Assume that $\Omega$ is closed, and there exists a constant $M_u>0$ such that $|u_\Omega(y)|\le M_u,$ for all $y\in\Omega$. Then,
% \begin{equation}\label{eq:truncation_error}
% \left|
% \mathbb E\!\left[
% W_K\,u_\Omega(Y_K)
% \right]
% \right|
% \le
% M_u\,\tau_W.
% \end{equation}
% \end{lemma}

% \begin{proof}
% By the stopping rule in the closed manifold case, we have $W_K<\tau_W$. Hence,
% \[
% \left|
% \mathbb E\!\left[
% W_K\,u_\Omega(Y_K)
% \right]
% \right|\le
% \mathbb E\!\left|
% W_K\,u_\Omega(Y_K)
% \right|
% \le
% \mathbb E\!\left[W_K\,|u_\Omega(Y_K)|\right]
% \le
% M_u\,\tau_W.
% \]
% This ends the proof.
% \end{proof}

%We finally estimate the error introduced by replacing the exact compensation term $m$ with its approximation $m_\ast$.
We finally estimate the error from replacing the exact compensation term $m$ with its approximation $m_\ast$.
\begin{lemma}\label{lem:compensation_error}
Assume that there exists a constant $\eta_m>0$ and $c_\ast\in(0,1)$ such that $|m(x)-m_\ast(x)|\le \eta_m$ for all $x\in U_\delta^\Omega$ and $c_{\sigma,r(y)}\le c_\ast$, for all $y\in\Omega_{\varepsilon}$. Then,
\begin{equation}\label{eq:compensation_error}
\bigg|\mathbb E\!\Big[\sum_{k=0}^{K-1}W_k\,\omega(Y_k)\,\bigl(m(Z_{k+1})-m_\ast(Z_{k+1})\bigr)\Big]\bigg|
\le \frac{\eta_m}{\sigma(1-c_\ast)}.
\end{equation}
\end{lemma}

\vspace{-4pt}
\begin{proof}
By the assumption on $m-m_\ast$, we have $|m(Z_{k+1})-m_\ast(Z_{k+1})|
\le \eta_m$. 
%Then, we derive from \eqref{eq:omega_explicit_formula}  that
%\[
%\omega(y)=\frac{1-c_{\sigma,r(y)}}{\sigma}\le \frac{1}{\sigma},
%\quad
%W_k=\prod_{j=0}^{k-1}c_{\sigma,r(Y_j)}\le c_\ast^k,
%\]
%with $c_\ast\in(0,1)$. These bounds imply
Combining this with \eqref{omegaboundx}, we obtain that
\begin{equation*}
\begin{split}
\bigg|\mathbb E\bigg[\;\sum_{k=0}^{K-1}W_k\,\omega(Y_k)\,\bigl(m(Z_{k+1})-m_\ast(Z_{k+1})\bigr)\;\bigg]\bigg|&\le\mathbb E\!\bigg[\;\sum_{k=0}^{K-1}W_k\,\omega(Y_k)\,\Big|m(Z_{k+1})-m_\ast(Z_{k+1})\Big|\bigg]\\
&\le
\frac{\eta_m}{\sigma}\sum_{k=0}^{K-1}c_\ast^k
\le
\frac{\eta_m}{\sigma}\sum_{k=0}^{\infty}c_\ast^k
=
\frac{\eta_m}{\sigma(1-c_\ast)}.
\end{split}
\end{equation*}
 In the manifold-with-boundary case, the constant in this estimate may depend on $\varepsilon$ through $c_\ast$.
This ends the proof.
\end{proof}

%\begin{remark}\label{rem:compensation_cases}
%\itshape
%Lemma~\ref{lem:compensation_error} measures the effect of the compensation term through the uniform approximation error $\eta_m:=\|m-m_\ast\|_{L^\infty(U_\delta^\Omega)}$. For the second-order compensation constructed above, one expects $m(x)-m_\ast(x)=O(|\eta(x)|^3)$, provided that the coefficients $m_{1,\ast}$ and $m_{2,\ast}$, computed by the normal finite-difference procedure, approximate $m_1$ and $m_2$ with sufficient accuracy. In this case, if $|\eta(x)|\le r_\ast$ on the sampled tubular neighborhood, then one formally obtains $\eta_m=O(r_\ast^3)$.
%
%In the baseline case $m_\ast\equiv0$, the above lemma reduces to the estimate with $\eta_m=\|m\|_{L^\infty(U_\delta^\Omega)}$. Since $m(y)=0$ on $\Omega$, the corresponding error is caused entirely by the off-manifold mismatch. Therefore, when the local sampling radius is sufficiently small, this contribution may remain negligible compared with the Monte Carlo error, whereas the second-order compensation is expected to further reduce this bias at the price of additional computation.
%\end{remark}
\begin{remark}\label{rem:compensation_cases}
\itshape
Lemma~{\rm\ref{lem:compensation_error}} quantifies the effect of the compensation term through the uniform approximation error
$\eta_m:=\|m-m_\ast\|_{L^\infty(U_\delta^\Omega)}$.
For the second-order compensation above, one expects
$m(x)-m_\ast(x)=O(|\eta(x)|^3)$, provided that the finite-difference approximations of $m_1$ and $m_2$ are sufficiently accurate. Hence, if $|\eta(x)|\le r_\ast$ in the sampled tubular neighborhood, then $\eta_m=O(r_\ast^3)$.

In the baseline case $m_\ast\equiv0$, the above lemma reduces to the estimate with $\eta_m=\|m\|_{L^\infty(U_\delta^\Omega)}$. Since $m(y)=0$ on $\Omega$, the resulting error is caused entirely by the off-manifold mismatch. Therefore, when the local sampling radius is small, this contribution may remain negligible compared with the Monte Carlo error, while the second-order compensation further reduces the bias at additional computational cost.
\end{remark}
 
When the closest-point map is evaluated numerically, the projected step is computed with an approximate map $\operatorname{cp}_h$ instead of $\operatorname{cp}$. We measure this projection error by
\begin{equation}\label{eq:cp_error}
\eta_{\operatorname{cp}}
:=
\sup_{x\in U_\delta^\Omega}
|\operatorname{cp}_h(x)-\operatorname{cp}(x)|.
\end{equation}
%Let $S(y_0)$ denote the path contribution generated by the exact projected chain, and let $S_h(y_0)$ denote the corresponding contribution obtained by replacing $\operatorname{cp}$ with $\operatorname{cp}_h$. 
We assume below that the quantities entering the path contribution depend Lipschitz continuously on the projected point, and that the projected chain is stable under the perturbation $\operatorname{cp}\mapsto\operatorname{cp}_h$.

\begin{assumption}\label{ass:projection_stability}
Let $S(y_0)$ and $S_h(y_0)$ denote the path contributions generated by the exact closest-point map $\operatorname{cp}$ and its numerical approximation $\operatorname{cp}_h$, respectively. We assume that the numerical projected recursion is stable under perturbations of the closest-point projection %, in the sense that
\begin{equation}\label{eq:path_contribution_stability_cp}
\big| \mathbb E[S_h(y_0)]-\mathbb E[S(y_0)] \big| \le C_{\operatorname{cp}}\,\eta_{\operatorname{cp}},
\end{equation}
where the constant $C_{\operatorname{cp}}>0$ is independent of $\eta_{\operatorname{cp}}$ and may depend on the boundary-layer parameter $\varepsilon$ through $\Omega_\varepsilon$ in the manifold-with-boundary case.
\end{assumption}

Now, we are ready to derive the error estimates for the projected Walk on Spheres method in the case of manifolds with boundary and in the closed manifold case. %We write $\bar{u}_{\Omega,N}^{h}(y_0):=\frac{1}{N}\sum_{i=1}^N S_h^i(y_0)$ for the estimator computed with the numerical closest-point map $\operatorname{cp}_h$.

\begin{theorem}\label{thm:total_error_boundary}
Assume that $\Omega$ has boundary $\Gamma$. Let
$\bar{u}_{\Omega,N}^{h}(y_0)$ be the numerical solution associated with closest-point map $\operatorname{cp}_h$, and let $u_\Omega(y_0)$ be the exact solution of \eqref{eq:surface_pde}.
 For any $\varepsilon>0$, under Assumptions \ref{ass:boundary_lipschitz} and \ref{ass:projection_stability}, we have
 \begin{equation}\label{eq:total_error_boundary}
\mathbb E\bigl[|\bar{u}_{\Omega,N}^{h}(y_0)-u_\Omega(y_0)|^2\bigr]\lesssim  N^{-1}+ \varepsilon^2+ \eta_m^2+ \eta_{\operatorname{cp}}^2.
\end{equation}
%\begin{equation}\label{eq:total_error_boundary}
%\mathbb E\bigl[
%|\bar{u}_{\Omega,N}^{h}(y_0)-u_\Omega(y_0)|^2
%\bigr]
%\le C_{\varepsilon}N^{-1}+C_\Gamma\varepsilon^2+C_{\varepsilon}\eta_m^2+C_{\operatorname{cp}}\eta_{\operatorname{cp}}^2,
%\end{equation}
%where the constant $C_{\varepsilon}$ is independent of $N$, $\eta_m$, and $\eta_{\operatorname{cp}}$, but may depend on $\varepsilon$ through $r_{\min}$ and $c_\ast$, while $C_\Gamma$ depends on the boundary Lipschitz constant $L_\Gamma$ in Assumption~\ref{ass:boundary_lipschitz}.
\end{theorem}

\vspace{-10pt}
\begin{proof}
By the same argument as in Lemma \ref{lem:variance_bound}, the variance estimate remains valid for the estimator computed with $\operatorname{cp}_h$. Hence, together with Lemma \ref{lem:boundary_layer_error}, Lemma \ref{lem:compensation_error}, and Assumption \ref{ass:projection_stability}, we obtain
\begin{eqnarray*}
 &&\hspace{-12pt}\mathbb E\bigl[
|\bar{u}_{\Omega,N}^{h}(y_0)-u_\Omega(y_0)|^2
\bigr]=\mathbb E\bigl[
|\bar{u}_{\Omega,N}^{h}(y_0)-\mathbb E[\bar{u}_{\Omega,N}^{h}(y_0)]
+\mathbb E[\bar{u}_{\Omega,N}^{h}(y_0)]-u_\Omega(y_0)|^2
\bigr]
\\&&\hspace{10pt}\le 2\,\mathbb E\bigl[
|\bar{u}_{\Omega,N}^{h}(y_0)-\mathbb E[\bar{u}_{\Omega,N}^{h}(y_0)]|^2
\bigr]+
2\,|\mathbb E[\bar{u}_{\Omega,N}^{h}(y_0)]-u_\Omega(y_0)|^2
\\&&\hspace{10pt}=
2\,\operatorname{Var}\bigl(\bar{u}_{\Omega,N}^{h}(y_0)\bigr)
+
2D_1,
\end{eqnarray*}
where $D_1:=|\mathbb E[\bar{u}_{\Omega,N}^{h}(y_0)]-u_\Omega(y_0)|^2$.

It remains to estimate $D_1$. Moreover, by adding and subtracting $\mathbb E[\bar{u}_{\Omega,N}(y_0)]$, and using the definition of $\bar{u}_{\Omega,N}(y_0)$ and \eqref{eq:n_step_representation} with $n=K$, we have
\begin{eqnarray*}
&&D_1=\big|\mathbb E[\bar{u}_{\Omega,N}^{h}(y_0)]-u_\Omega(y_0)\big|^2\\
&&\le
\bigg\{\big|\mathbb E[\bar{u}_{\Omega,N}^{h}(y_0)]-\mathbb E[\bar{u}_{\Omega,N}(y_0)]\big|+
\big|
\mathbb E[\bar{u}_{\Omega,N}(y_0)]-u_\Omega(y_0)\big|\bigg\}^2\\
&&=
\bigg\{\Big|\mathbb E[S_h(y_0)]-\mathbb E[S(y_0)]\Big|+\bigg|
\mathbb E\!\big[
W_K\,g_\Omega\!\big(\operatorname{cp}_\Gamma(Y_K)\big)+
\mathbb E\!\bigg[
\sum_{k=0}^{K-1}
W_k\,\omega(Y_k)\big(f(Z_{k+1})+m_\ast(Z_{k+1})\big)
\bigg]
\bigg]\\
&&\hspace{12pt}-
\mathbb E\!\big[
W_K\,u_\Omega(Y_K)
\big]
-
\mathbb E\!\bigg[
\sum_{k=0}^{K-1}
W_k\,\omega(Y_k)\big(f(Z_{k+1})+m(Z_{k+1})\big)
\bigg]
\bigg|
\bigg\}^2\\
&&\le\bigg\{
\big|
\mathbb E[S_h(y_0)]
-
\mathbb E[S(y_0)]
\big|
+
\bigg|
\mathbb E\!\big[
W_K\,g_\Omega\!\big(\operatorname{cp}_\Gamma(Y_K)\big)
\big]
-
\mathbb E\!\big[
W_K\,u_\Omega(Y_K)
\big]
\bigg|\\
&&\hspace{12pt}
+
\bigg|
\mathbb E\!\big[
\sum_{k=0}^{K-1}
W_k\,\omega(Y_k)\big(m_\ast(Z_{k+1})-m(Z_{k+1})\big)
\big]
\bigg|
\bigg\}^2\\
&&\le
\Big(C_{cp}\,\eta_{\operatorname{cp}}+L_\Gamma\,\varepsilon+\frac{\eta_m}{\sigma(1-c_\ast)}\Big)^2.
\end{eqnarray*}

Combining the above estimates yields
\begin{equation*}
\begin{split}
\mathbb E\bigl[
|\bar{u}_{\Omega,N}^{h}(y_0)-u_\Omega(y_0)|^2
\bigr]
&\le
2\,\mathcal O(N^{-1})
+
2\Big(
C_{cp}\,\eta_{\operatorname{cp}}
+
L_\Gamma\,\varepsilon
+
\frac{\eta_m}{\sigma(1-c_\ast)}
\Big)^2\\
&\le C_{\varepsilon}N^{-1}
+
C_\Gamma\varepsilon^2
+
C_{\varepsilon}\eta_m^2
+
C_{cp}\,\eta_{\operatorname{cp}}^2.
\end{split}
\end{equation*}
This ends the proof.
\end{proof}

We next derive the corresponding error estimate in the closed manifold case.

\begin{theorem}\label{thm:total_error_boundaryless}
Assume that $\Omega$ is closed. Let
$\bar{u}_{\Omega,N}^{h}(y_0)$ be the numerical solution associated with closest-point map $\operatorname{cp}_h$, and let $u_\Omega(y_0)$ be the exact solution of \eqref{eq:surface_pde}. For any $\tau_W>0$, there exists a constant $M_u>0$ such that $|u_\Omega(y)|\le M_u$, for all $y\in\Omega$.
% let $u_\Omega$ satisfy $|u_\Omega(y)|\le M_u$ for all $y\in\Omega$. 
Under the Assumption \ref{ass:projection_stability}, we have
\begin{equation}\label{eq:total_error_boundaryless}
\mathbb E\bigl[
|\bar{u}_{\Omega,N}^{h}(y_0)-u_\Omega(y_0)|^2
\bigr]\lesssim N^{-1}+\tau_W^2+\eta_m^2+\eta_{\operatorname{cp}}^2.
\end{equation}
\end{theorem}

\begin{proof}
By the stopping rule in the closed manifold case, we have $W_K<\tau_W$. Hence,
\begin{equation}\label{tau_w_bound}
\left|
\mathbb E\!\left[
W_K\,u_\Omega(Y_K)
\right]
\right|\le
\mathbb E\!\left|
W_K\,u_\Omega(Y_K)
\right|
\le
\mathbb E\!\left[W_K\,|u_\Omega(Y_K)|\right]
\le
M_u\,\tau_W.
\end{equation}
By the same argument as in Lemma \ref{lem:variance_bound}, the variance estimate remains valid for the estimator computed with $\operatorname{cp}_h$. Hence, together with \eqref{tau_w_bound}, Lemma \ref{lem:compensation_error}, and Assumption \ref{ass:projection_stability}, we obtain
\begin{eqnarray*}
&&\mathbb E\bigl[
|\bar{u}_{\Omega,N}^{h}(y_0)-u_\Omega(y_0)|^2
\bigr]=\mathbb E\bigl[|\bar{u}_{\Omega,N}^{h}(y_0)-\mathbb E[\bar{u}_{\Omega,N}^{h}(y_0)]
+\mathbb E[\bar{u}_{\Omega,N}^{h}(y_0)]-u_\Omega(y_0)|^2
\bigr]\\
&&\hspace{18pt}\le
2\,\mathbb E\bigl[
|\bar{u}_{\Omega,N}^{h}(y_0)-\mathbb E[\bar{u}_{\Omega,N}^{h}(y_0)]|^2
\bigr]+
2\,|\mathbb E[\bar{u}_{\Omega,N}^{h}(y_0)]-u_\Omega(y_0)|^2\\
&&\hspace{18pt}=
2\,\operatorname{Var}\bigl(\bar{u}_{\Omega,N}^{h}(y_0)\bigr)
+
2D_2,
\end{eqnarray*}
where $D_2:=|\mathbb E[\bar{u}_{\Omega,N}^{h}(y_0)]-u_\Omega(y_0)|^2$. Moreover, by adding and subtracting $\mathbb E[\bar{u}_{\Omega,N}(y_0)]$, and using the definition of $\bar{u}_{\Omega,N}(y_0)$ and \eqref{eq:n_step_representation} with $n=K$, we have
\begin{eqnarray*}
&&D_2=\big|\mathbb E[\bar{u}_{\Omega,N}^{h}(y_0)]-u_\Omega(y_0)\big|^2\\
&&\hspace{15.5pt}\le\bigg\{\big|\mathbb E[\bar{u}_{\Omega,N}^{h}(y_0)]-\mathbb E[\bar{u}_{\Omega,N}(y_0)]\big|+\big|
\mathbb E[\bar{u}_{\Omega,N}(y_0)]-u_\Omega(y_0)\big|\bigg\}^2\\
&&\hspace{15.5pt}=\bigg\{\big|\mathbb E[S_h(y_0)]-\mathbb E[S(y_0)]\big|+\bigg|\mathbb E\!\bigg[\sum_{k=0}^{K-1}W_k\,\omega(Y_k)\big(f(Z_{k+1}+m_\ast(Z_{k+1})\big)\bigg]\\
&&\hspace{27.5pt}-\mathbb E\!\big[W_K\,u_\Omega(Y_K)\big]-\mathbb E\bigg[\sum_{k=0}^{K-1}W_k\,\omega(Y_k)\big(f(Z_{k+1})+m(Z_{k+1})\big)\bigg]\bigg|
\bigg\}^2\\
&&\hspace{15.5pt}\le\bigg\{\big|\mathbb E[S_h(y_0)]-\mathbb E[S(y_0)]\big|+\Big|\mathbb E\!\big[W_K\,u_\Omega(Y_K)\big]\Big|\\
&&\hspace{27.5pt}+\bigg|\mathbb E\bigg[\sum_{k=0}^{K-1}W_k\,\omega(Y_k)\big(m_\ast(Z_{k+1})-m(Z_{k+1})\big)\bigg]\bigg|\bigg\}^2\\
&&\hspace{15.5pt}\le\Big(C_{cp}\,\eta_{\operatorname{cp}}+M_u\,\tau_W+\frac{\eta_m}{\sigma(1-c_\ast)}\Big)^2.
\end{eqnarray*}
Therefore, combining the above estimates yields
\begin{equation*}
\begin{split}
\mathbb E\bigl[
|\bar{u}_{\Omega,N}^{h}(y_0)-u_\Omega(y_0)|^2
\bigr]
&\le
2\,\mathcal O(N^{-1})
+
2\Big(C_{cp}\,\eta_{\operatorname{cp}}+M_u\,\tau_W+\frac{\eta_m}{\sigma(1-c_\ast)}
\Big)^2\\
&\le
\mathcal O\!\big(
N^{-1}+\tau_W^2+\eta_m^2+\eta_{\operatorname{cp}}^2
\big).
\end{split}
\end{equation*}
This ends the proof.
\end{proof}

% \begin{remark}\label{rem:projection_error_path}
% \itshape
% The above error estimates are derived under the assumption that the closest-point map $\operatorname{cp}$ in the projected chain is evaluated exactly. In practice, however, the recursion step $Y_{k+1}=\operatorname{cp}(\Theta_{k+1})$ is replaced by a numerical projection $Y_{k+1}^h=\operatorname{cp}_h(\Theta_{k+1})$, where $\operatorname{cp}_h$ denotes an approximate closest-point solver. Define
% \[
% \eta_{\operatorname{cp}}
% :=
% \sup_{x\in U_\delta^\Omega}
% |\operatorname{cp}_h(x)-\operatorname{cp}(x)|.
% \]
% Then the path construction introduces an additional projection error. Under suitable regularity assumptions on $u_\Omega$, $f_\Omega$, and, in the case of manifolds with boundary, $g_\Omega$, this contribution is expected to be of order $O(\eta_{\operatorname{cp}})$ in bias, and hence of order $O(\eta_{\operatorname{cp}}^2)$ in the mean-square error. Therefore, if the numerical closest-point map is used, the final estimates in Theorems \ref{thm:total_error_boundary} and \ref{thm:total_error_boundaryless} should be supplemented by an additional projection term.
% \end{remark}

\section{Numerical Experiments}

In this section, several numerical experiments are performed to demonstrate the accuracy and efficiency of the proposed projected Walk on Spheres method for solving \eqref{eq:surface_pde} on embedded manifolds. %All experiments were performed on a personal laptop (Lenovo Legion Y9000P, model 82WK), equipped with an Intel(R) Core(TM) i9-13900HX processor (24 cores, 32 threads, base frequency 2.2 GHz), 16 GB RAM, running Microsoft Windows 11 Home (version 10.0.26100, build 26100), BIOS version KWCN38WW (UEFI mode). 
All numerical experiments were performed on a personal laptop equipped with an Intel Core i9-13900HX processor and 16 GB RAM running Windows 11, and all codes were implemented in MATLAB R2023a with parallel computing enabled.

In our computations, we use the following root mean square error to measure the numerical accuracy:
\begin{equation}\label{eq:rmse_section}
\mathrm{Error}
=
\bigg(
\frac{1}{N_s}
\sum_{p=1}^{N_s}
\bigl(u(x_p)-u_N(x_p)\bigr)^2
\bigg)^{1/2},
\end{equation}
where $u(x_p)$ and $u_N(x_p)$ denote the exact and numerical solutions at point $x_p$, respectively, and $\{x_p\}_{p=1}^{N_s}$ denotes the set of random points on the manifold $\Omega$. In what follows, we let $N_s=100$ for all experiments and choose the boundary layer $\varepsilon=10^{-5}$ in the case of manifolds with boundary. The number of samples per path is set to $N_V=50$ for all examples, and the truncation threshold is set to $\tau_W=10^{-8}$ in the closed manifold case.

\vspace{2pt}

\textbf{Example 1} {\bf (Effect of compensation term on a torus with a planar cut).}
In this example, we investigate the effect of the compensation term on the accuracy of the proposed projected Walk on Spheres method. We consider a torus embedded in $\mathbb{R}^3$, given by the parametrization
\[
X(\theta,\phi)
=
\bigl(
(R+a\cos\theta)\cos\phi,\,
(R+a\cos\theta)\sin\phi,\,
a\sin\theta
\bigr),
\quad
\theta,\phi\in[0,2\pi),
\]
with $R=1$ and $a=0.45$. To introduce a boundary, we cut the torus by the plane $-\sin(\phi_c)x+\cos(\phi_c)y=0$ with $\phi_c=0.35\pi$, so that the computational domain is given by $\phi-\phi_c\in(0,\pi)$.

We consider the screened Poisson equation \eqref{eq:surface_pde} with $\sigma=4$, and the exact solution is set as
\[
u_\Omega(\theta,\phi)
=
\sin(s)\,A(\theta)
+
\sin(2s)\,B(\theta)
+
\sin(3s)\,C(\theta),
\quad
s=\phi-\phi_c,
\]
where
\[
\begin{aligned}
A(\theta)&=1+0.90\cos\theta+0.60\sin(2\theta)+0.35\cos(3\theta),\\
B(\theta)&=0.45\cos(2\theta),\;\;\;C(\theta)=0.30\bigl(0.5+\sin\theta\bigr).
\end{aligned}
\]
The source term $f_\Omega$ is computed analytically via the Laplace--Beltrami operator, and the solution satisfies homogeneous Dirichlet boundary conditions $u_\Omega=0$ on $\partial\Omega$ by construction.
%and the source term $f_\Omega$ can be computed analytically from the Laplace--Beltrami operator on the manifold. By construction, the exact solution satisfies $u_\Omega=0$ on $\partial\Omega$, and therefore the problem is equipped with homogeneous Dirichlet boundary conditions.

\begin{wrapfigure}{r}{0.50\textwidth}
\vspace{-16pt} 
\centering
\includegraphics[width=0.46\textwidth]{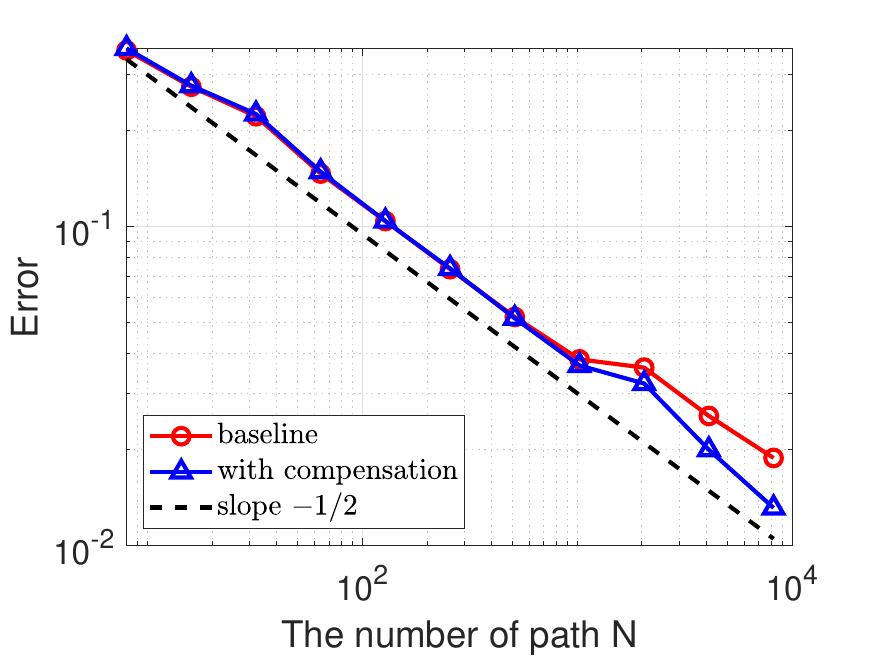}
\vspace{2pt}\caption{RMSE comparison between the baseline method and the compensated method on the torus with a planar cut.}
\label{fig:torus_comp}
\vspace{-2pt}
\end{wrapfigure}
% \begin{wrapfigure}{r}{0.50\textwidth}
% \vspace{-18pt}
% \makebox[\linewidth][r]{\includegraphics[width=0.96\linewidth]{figure/torus_compare_m.pdf}}
% \refstepcounter{figure}\label{fig:torus_comp}
% \makebox[\linewidth][r]{\parbox{\linewidth}{\raggedright\small
% \vspace{-3pt} \textsc{Figure}~\thefigure. RMSE comparison between\\
% the baseline method and the compensated\\
% method on the torus with a planar cut.}}

% \end{wrapfigure}

%\begin{wrapfigure}{r}{0.6\textwidth}
%\vspace{-16pt}
%\hspace*{5mm}
%\centering
%\includegraphics[width=0.80\linewidth]{figure/torus_compare_m.pdf}
%\hspace*{5mm}\caption{RMSE comparison between the baseline method and the compensated method on the torus with a planar cut.}
%\label{fig:torus_comp}
%\vspace{-2pt}
%\end{wrapfigure}

To evaluate the effect of the compensation term, we compare the baseline method (i.e., $m_\ast\equiv 0$) with a second-order compensated version in which $m_\ast$ is constructed numerically from a coarse Monte Carlo approximation.
Fig.~\ref{fig:torus_comp} shows that the baseline method follows the expected $\mathcal{O}(N^{-1/2})$ convergence for moderate $N$, but eventually saturates due to the bias. In contrast, the compensated method maintains the $\mathcal{O}(N^{-1/2})$ decay throughout, demonstrating the effectiveness of the second-order correction.

To further illustrate the impact of the compensation term, we examine the spatial distribution of the numerical solution and the corresponding pointwise errors obtained with $N=2000$.
\begin{figure}[ht!]
\centering
\includegraphics[width=0.36\textwidth,trim=50 70 5 10,clip]{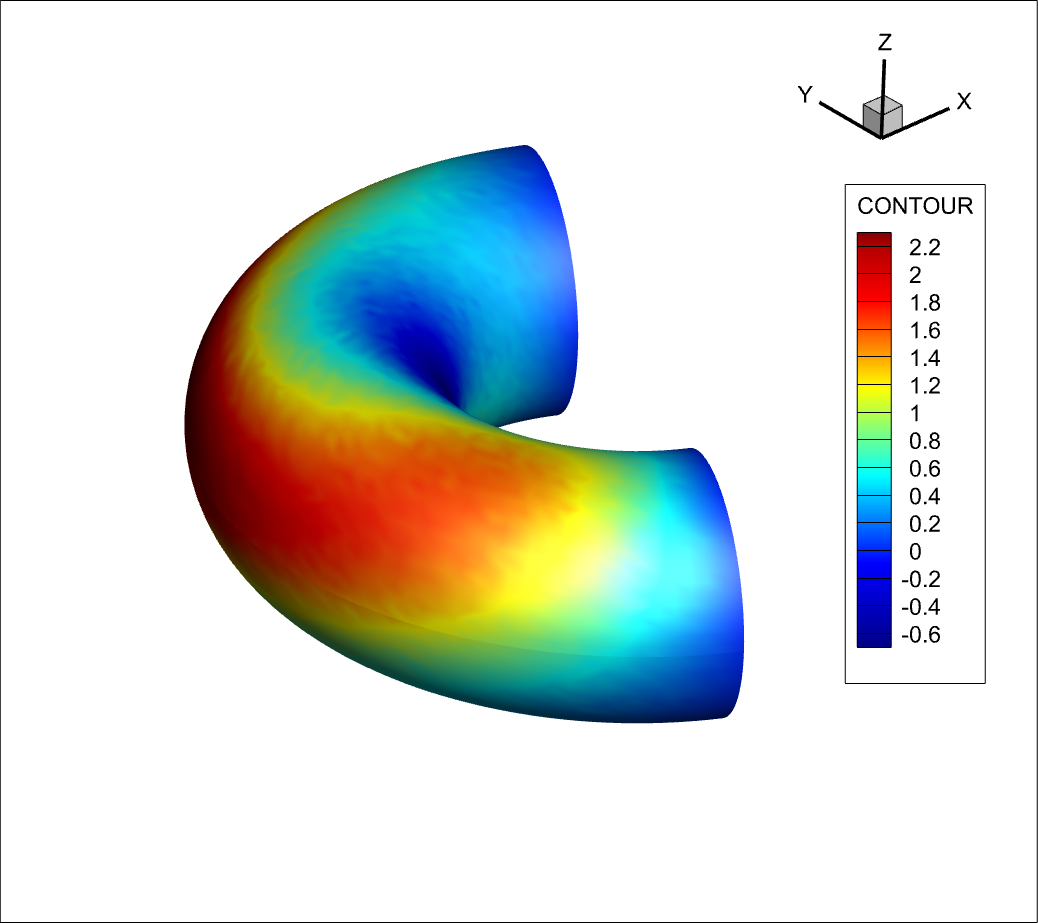}
\includegraphics[width=0.36\textwidth,trim=50 70 5 10,clip]{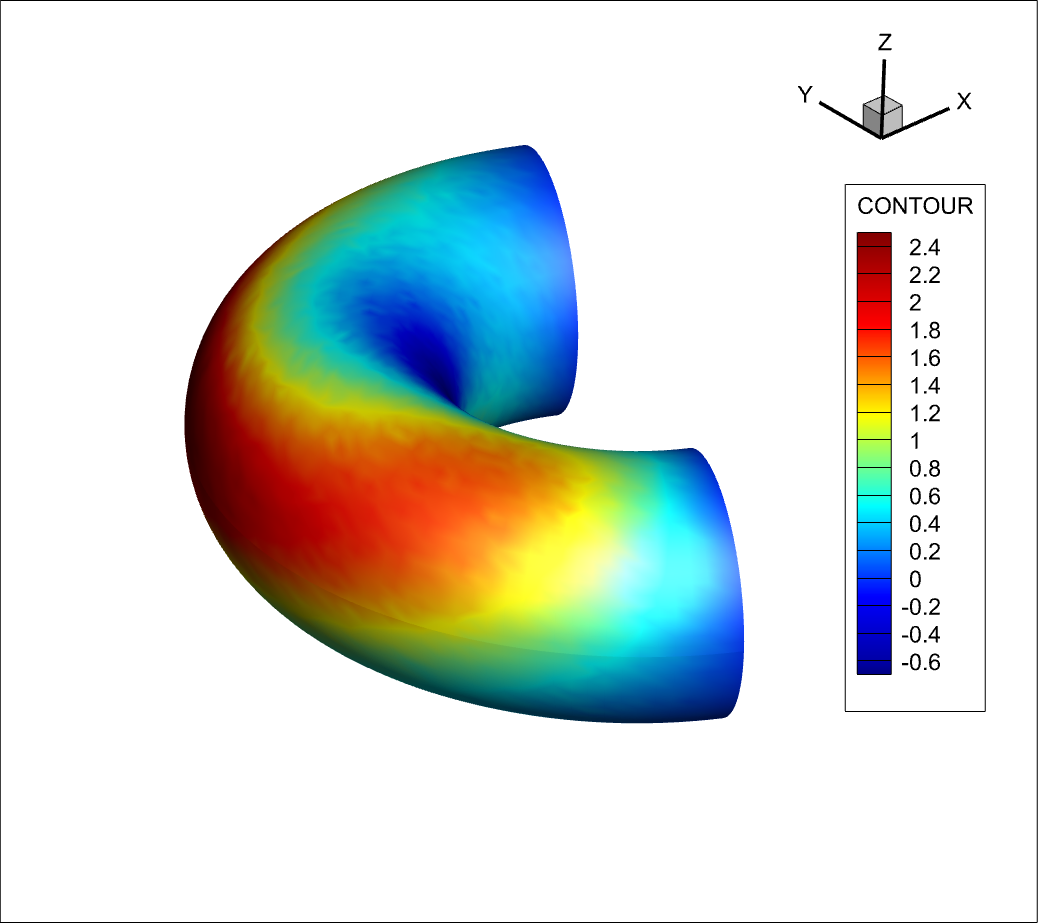}

\vspace{-16pt}
\includegraphics[width=0.36\textwidth,trim=50 70 5 10,clip]{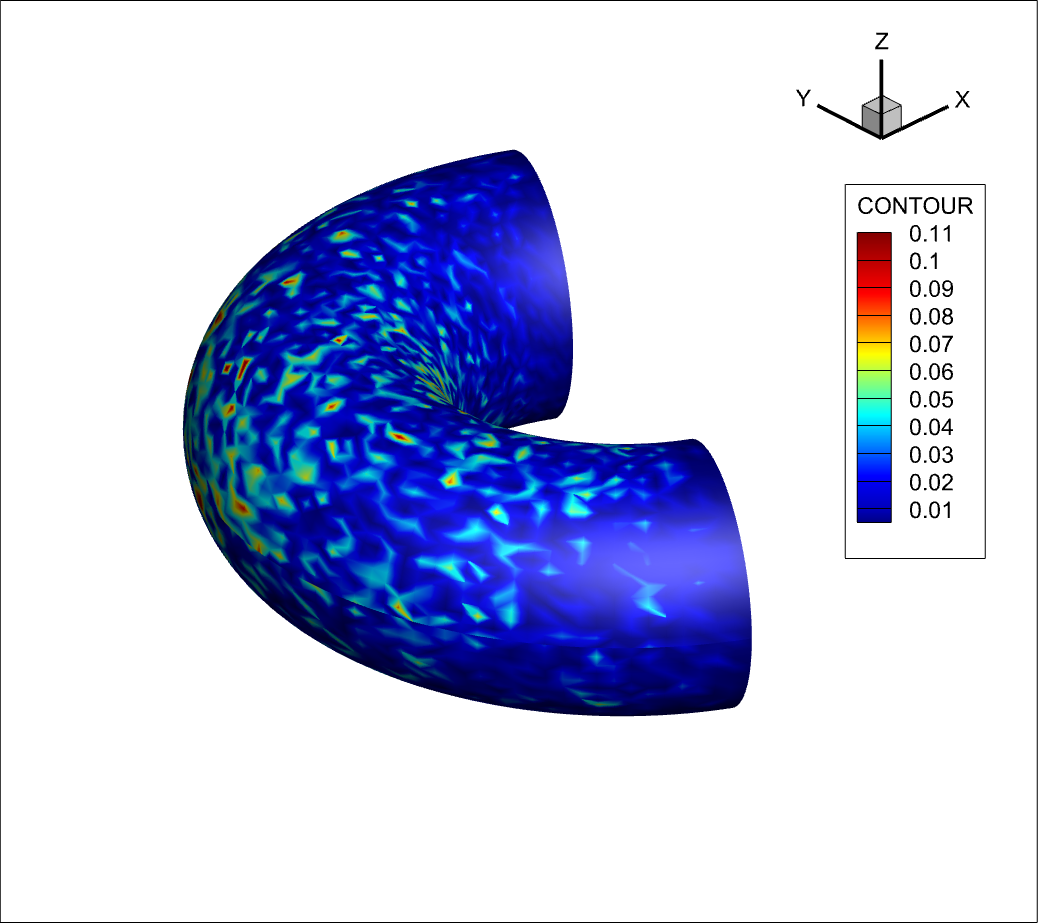}
\includegraphics[width=0.36\textwidth,trim=50 70 5 10,clip]{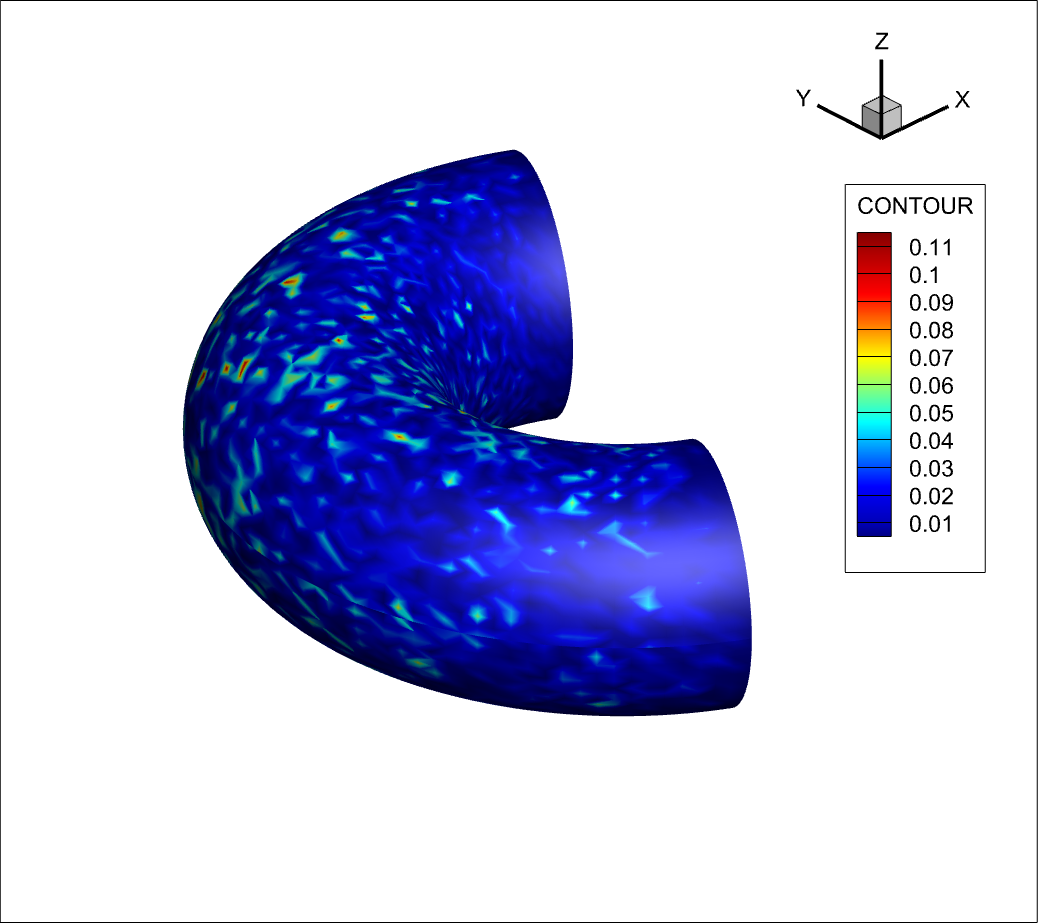}

\caption{Top: numerical solutions without (left) and with (right) compensation. 
Bottom: absolute errors without (left) and with (right) compensation.}
\label{fig:torus_solution_error}
\end{figure}
As shown in Fig.~\ref{fig:torus_solution_error}, the two numerical solutions are nearly indistinguishable and both accurately capture the solution. However, a clear difference is observed in the error distributions. Without compensation, the error exhibits a pronounced spatial pattern, with larger deviations appearing away from the boundary. In contrast, the compensated method yields a substantially smaller error and a more uniform distribution, indicating that the second-order correction effectively reduces the bias. %This example indicates that the compensation term reduces the operator mismatch error at the cost of additional normal-direction sampling.  We therefore use the baseline method with $m_\ast\equiv0$ in the remaining examples to reduce the computational cost while retaining the mesh-free nature of the method.

%Example 1 shows that the compensation term mainly reduces the bias caused by the mismatch between the ambient operator and the intrinsic surface operator, at the expense of additional computational cost due to extra sampling in the normal direction. In the following examples, we therefore focus on the baseline method with $m_\ast \equiv 0$. 
%% This choice is motivated by the fact that, for sufficiently small local ball radii, the mismatch error is negligible relative to the stochastic error, in agreement with the discussion in Remark~\ref{rem:m0}. 
% As discussed in Section~\ref{sec:M}, this choice avoids the construction of an auxiliary predictor and the reconstruction of normal derivatives. For sufficiently small local ball radii, the off-manifold mismatch error is expected to remain negligible relative to the error.
%Omitting $m_\ast$ therefore reduces the computational cost per path while preserving the simplicity and mesh-free character of the method.

\vspace{2pt}

\textbf{Example 2} {\bf (Accuracy test on a 1000 dimensional implicit dumbbell hypersurface).}
In this example, we test the accuracy of the proposed projected Walk on Spheres method on a high dimensional closed dumbbell-type hypersurface given in implicit form. Let $x=(\hat x,x_\ast)\in\mathbb{R}^{1001},$ $\hat x\in\mathbb{R}^{1000},$ $x_\ast\in\mathbb{R},$
and define the manifold by
\[
\mathcal M_{1000}=\{x\in\mathbb R^{1001}:\Phi(x)=0\},
\qquad
\Phi(x)=\|\hat x\|_2-R(x_\ast),
\]
where
\[
R(x_\ast)=\sqrt{1-(x_\ast/\kappa)^2}\,
\bigl(1-\alpha+2\alpha(x_\ast/\kappa)^2\bigr),
\qquad
\alpha=0.35,\ \kappa=1.20.
\]
Geometrically, $\mathcal M_{1000}$ is a 1000 dimensional dumbbell hypersurface embedded in $\mathbb R^{1001}$.

We solve \eqref{eq:surface_pde} on the closed manifold $\Omega=\mathcal M_{1000}$ with $\sigma=2d$ and the exact solution is chosen as
\[
u_\Omega(x)
=
0.15\,x_\ast
+
0.05\,x_\ast^2
+
\beta_1\,v_1^\top \hat x
+
\beta_2\,\hat x^\top A_2 \hat x,
\quad x\in\mathcal M_{1000},
\]
where $\beta_1=1$, $\beta_2=0.20$, and
\[
v_1=\frac{1}{\sqrt {1000}}(1,1,\dots,1)^\top\in\mathbb R^{1000},
\qquad
A_2=\operatorname{diag}(a_1,\dots,a_{1000}),
\qquad
\sum_{j=1}^{1000} a_j=0.
\]
The source term $f_\Omega$ is computed analytically from the Laplace--Beltrami operator on the manifold.

\begin{wrapfigure}{r}{0.50\textwidth}
\vspace{-16pt} 
\centering
\includegraphics[width=0.46\textwidth]{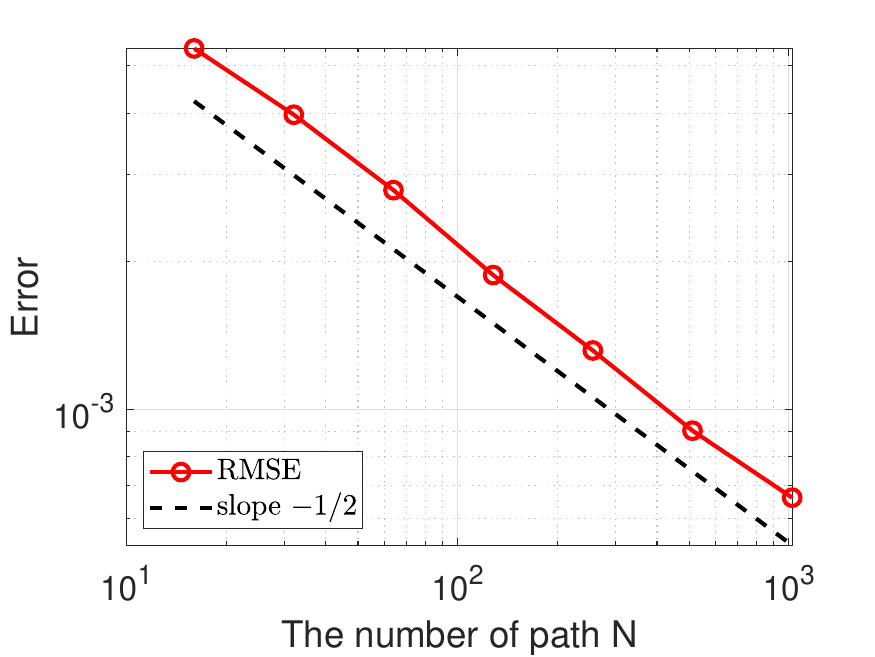}
\vspace{2pt}\caption{Numerical errors against the number of paths on the 1000-dimensional implicit dumbbell hypersurface.}
\label{fig:dumbbell_comp}
\vspace{-2pt}
\end{wrapfigure}

Since $\Omega$ is closed, the projected chain has no boundary stopping term and is terminated when the screened weight becomes sufficiently small. For simplicity, we set the compensation term to zero. The local sampling radius is chosen according to \eqref{eq:radius_rule_unified} and \eqref{eq:rho_geo_curvature}; in this example, we take $\theta=0.95$ and $c_\kappa=1.20$. For a point $x\in\mathbb R^{1001}$ near $\mathcal M_{1000}$, the closest-point projection is computed by a local Newton iteration in the profile parameter. To improve numerical stability in this high dimensional example, the local screened factor is evaluated by its Bessel-free series representation, and the local source contribution is approximated by a small-ball average.
As shown in Fig.~\ref{fig:dumbbell_comp}, the numerical error decreases as the number of paths increases and follows the reference slope $-1/2$ overall. This indicates that the MPWoS achieves the expected convergence rate and remains effective for high dimensional embedded manifolds.

% \begin{wrapfigure}{r}{0.50\textwidth}
% \vspace{-20pt}
% \makebox[\linewidth][r]{\includegraphics[width=0.9\linewidth]{figure/dumbbell_1000d.pdf}}
% \refstepcounter{figure}\label{fig:dumbbell_comp}
% \vspace{4pt}
% \makebox[\linewidth][r]{\parbox{\linewidth}{\raggedright\small
% \textsc{Figure}~\thefigure. Numerical errors against the\\
% number of paths on the 1000-dimensional\\
% implicit dumbbell hypersurface.}}
% \vspace{-10pt}
% \end{wrapfigure}

Table~\ref{tab:dumbbell_runtime} reports representative runtime statistics for the 1000-dimensional closed dumbbell hypersurface. The computation is carried out for single point evaluation, and the reported running time is averaged over 20 randomly selected test points. In all cases, the number of Green-density samples per step is fixed at $N_V=50$, and 24 CPU cores are used in the parallel implementation. The average number of recursion steps is nearly unchanged as $N$ varies, since it is primarily determined by the screened weight tolerance, the radius rule, and the geometry. The total running time therefore mainly reflects the number of independent Monte Carlo paths. The runtime also depends on the implementation, the available parallel resources, and the cost of the closest-point projection. In particular, manifolds with simpler projection formulas generally lead to a lower computational cost per path.

%\begin{table}[htbp]
%\centering
%\caption{Runtime statistics on the closed 1000-dimensional dumbbell.}
%\label{tab:dumbbell_runtime}
%\begin{tabular}{ccccc}
%\hline
%$N$ & $N_V$ & Average steps & Total time (s) & Time per path (ms) \\
%\hline
%100  & 50 & 203 & 1.65  & 16.50 \\
%500  & 50 & 205 & 5.70 & 11.40 \\
%1000 & 50 & 204 & 11.43 & 11.43 \\
%\hline
%\end{tabular}
%\end{table}
 \begin{table}[htbp]
\centering
\caption{Runtime statistics on the closed 1000-dimensional dumbbell.}
\label{tab:dumbbell_runtime}
\renewcommand{\arraystretch}{1.1}
\setlength{\tabcolsep}{11pt}
\begin{tabular}{ccccc}
\hline
$N$ & $N_V$ & Average steps & Total time (s) & Time per path (ms) \\
\hline
100  & 50 & 203 & 1.65  & 16.50 \\
500  & 50 & 205 & 5.70  & 11.40 \\
1000 & 50 & 204 & 11.43 & 11.43 \\
\hline
\end{tabular}
\end{table}

\textbf{Example 3} {\bf (Accuracy test on 3D manifold in $\mathbb R^9$).}
In this example, we test the accuracy of the proposed projected Walk on Spheres method on the closed manifold
\[
\mathcal M=S^1_{\mathrm{Fourier}}(R)\times S^2(a)\subset\mathbb R^9,
\quad\text{with}\; R=1,\ a=0.65,
\]
where $S^1_{\mathrm{Fourier}}(R)$ denotes the Fourier embedding of a circle of scale $R$ into $\mathbb R^6$, defined by
\[
C(\phi)=
\bigl(
\cos\phi,\,
\sin\phi,\,
\tfrac12\cos 2\phi,\,
\tfrac12\sin 2\phi,\,
\tfrac13\cos 3\phi,\,
\tfrac13\sin 3\phi
\bigr),
\quad \phi\in[0,2\pi),
\]
and $S^2(a)$ is the sphere of radius $a$ in $\mathbb R^3$. The manifold is parameterized by
$$X(\phi,\omega)=\begin{bmatrix}
R\,C(\phi)\\
a\,\omega
\end{bmatrix},\quad \omega\in S^2.
$$

%We solve \eqref{eq:surface_pde} on $\Omega=\mathcal M$, with give the exact solution
%$$u_\Omega(\phi,\omega)=\cos\phi\,\omega_1+0.30\sin(2\phi)\,\omega_2+0.20\cos(3\phi)\,(\omega_1^2-\omega_2^2),$$
%and the corresponding source term is computed exactly from the Laplace--Beltrami operator.

We solve \eqref{eq:surface_pde} on $\Omega=\mathcal M$ with exact solution $$u_\Omega(\phi,\omega)=\cos\phi\,\omega_1+0.30\sin(2\phi)\,\omega_2+0.20\cos(3\phi)\,(\omega_1^2-\omega_2^2),$$ the source term is computed exactly via the Laplace--Beltrami operator, and the compensation term is set to zero. 
For the closest-point projection, we treat the two factors separately. The first six coordinates are projected onto the Fourier curve by a local Newton iteration, while the last three coordinates are projected onto $S^2(a)$ by normalization.
As shown in Fig.~\ref{fig:S1Fourier_S2_embedded_R9}, the numerical error decreases steadily with respect to the number of paths, and the overall decay is consistent with the reference slope $-1/2$. This shows that the proposed projected Walk on Spheres method remains effective for low-dimensional manifolds embedded in high-dimensional spaces and exhibits the expected Monte Carlo convergence behavior.
%is given analytically by
%\[
%f_\Omega=
%\begin{bmatrix}
%\sigma+\dfrac{1}{3R^2}+\dfrac{2}{a^2}\\[6pt]
%\sigma+\dfrac{4}{3R^2}+\dfrac{2}{a^2}\\[6pt]
%\sigma+\dfrac{9}{3R^2}+\dfrac{6}{a^2}
%\end{bmatrix}
%^{\!\top}
%\begin{bmatrix}
%\cos\phi\,\omega_1\\[4pt]
%0.30\sin(2\phi)\,\omega_2\\[4pt]
%0.20\cos(3\phi)\,(\omega_1^2-\omega_2^2)
%\end{bmatrix}.
%\]

% \begin{wrapfigure}{r}{0.50\textwidth}
% \vspace{-22pt}
% \makebox[\linewidth][r]{\includegraphics[width=0.96\linewidth]{figure/S1Fourier_S2_embedded_R9.pdf}}
% \refstepcounter{figure}\label{fig:S1Fourier_S2_embedded_R9}
% \vspace{4pt}
% \makebox[\linewidth][r]{\parbox{\linewidth}{\raggedright\small
% \textsc{Figure}~\thefigure. Numerical errors against the number of paths.}}
% \vspace{-10pt}
% \end{wrapfigure}

\begin{wrapfigure}{r}{0.50\textwidth}
\vspace{-14pt} 
\centering
\includegraphics[width=0.46\textwidth]{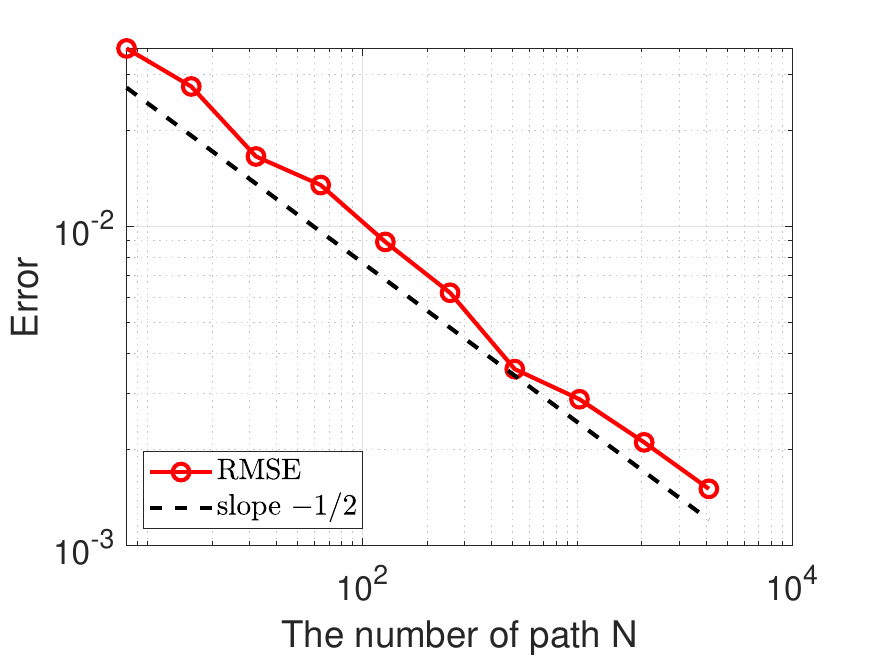}
\vspace{2pt}\caption{Numerical errors against the number of paths.}
\label{fig:S1Fourier_S2_embedded_R9}
\vspace{-2pt}
\end{wrapfigure}

%Again, we set the compensation term to zero in this example.  

To further illustrate the accuracy of the proposed method, we compare the exact solution, numerical solution, and absolute error via projection onto the $(x_1,x_2,x_7)$ coordinates with $N=500$. As shown in Fig.~\ref{fig:S1Fourier_S2_projection}, the numerical solution reproduces the main distribution of the exact solution well in the projected $(x_1,x_2,x_7)$ coordinates. In particular, the locations of high and low value regions are consistent in both plots. Moreover, the absolute error remains small throughout the domain, without noticeable concentration or artificial oscillations, indicating that the proposed method provides an accurate and stable approximation for this manifold embedded in $\mathbb{R}^9$.

\begin{figure}[ht!]
\hspace*{-0.3cm}
\includegraphics[width=0.32\textwidth,trim=20 50 20 20,clip]{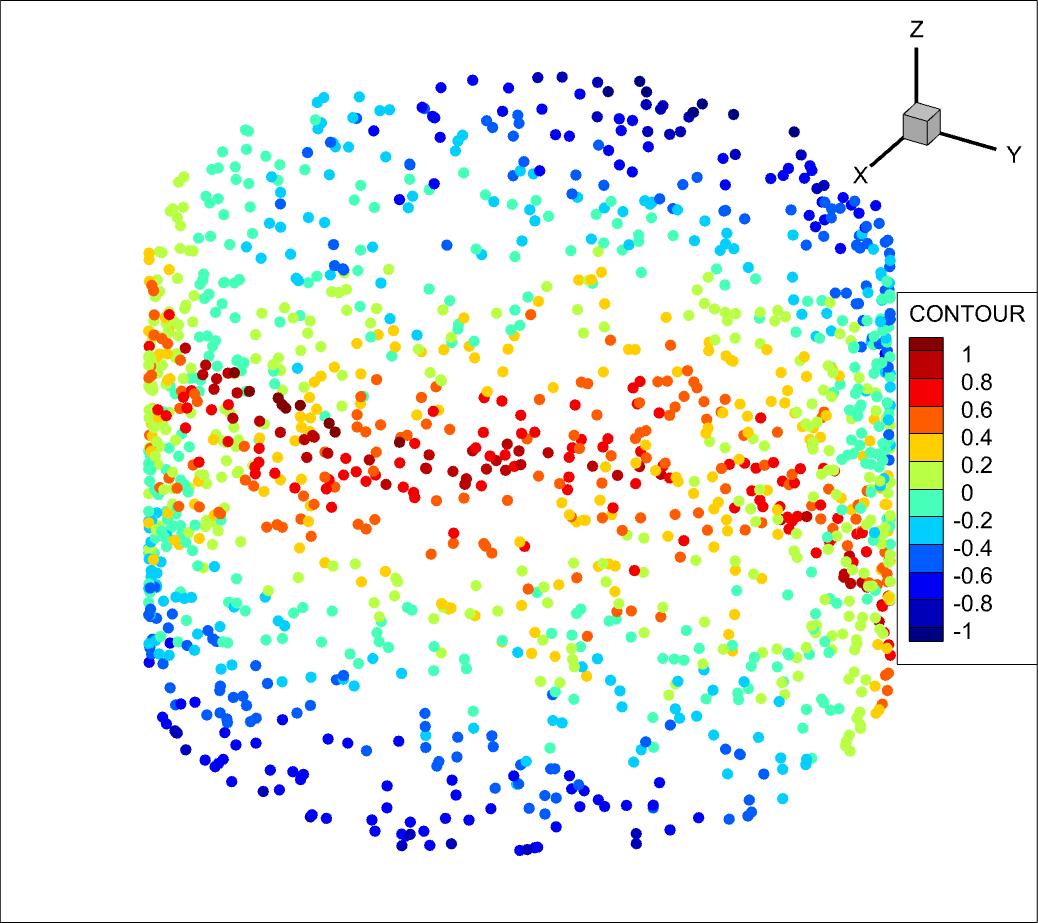}
\includegraphics[width=0.32\textwidth,trim=20 50 20 20,clip]{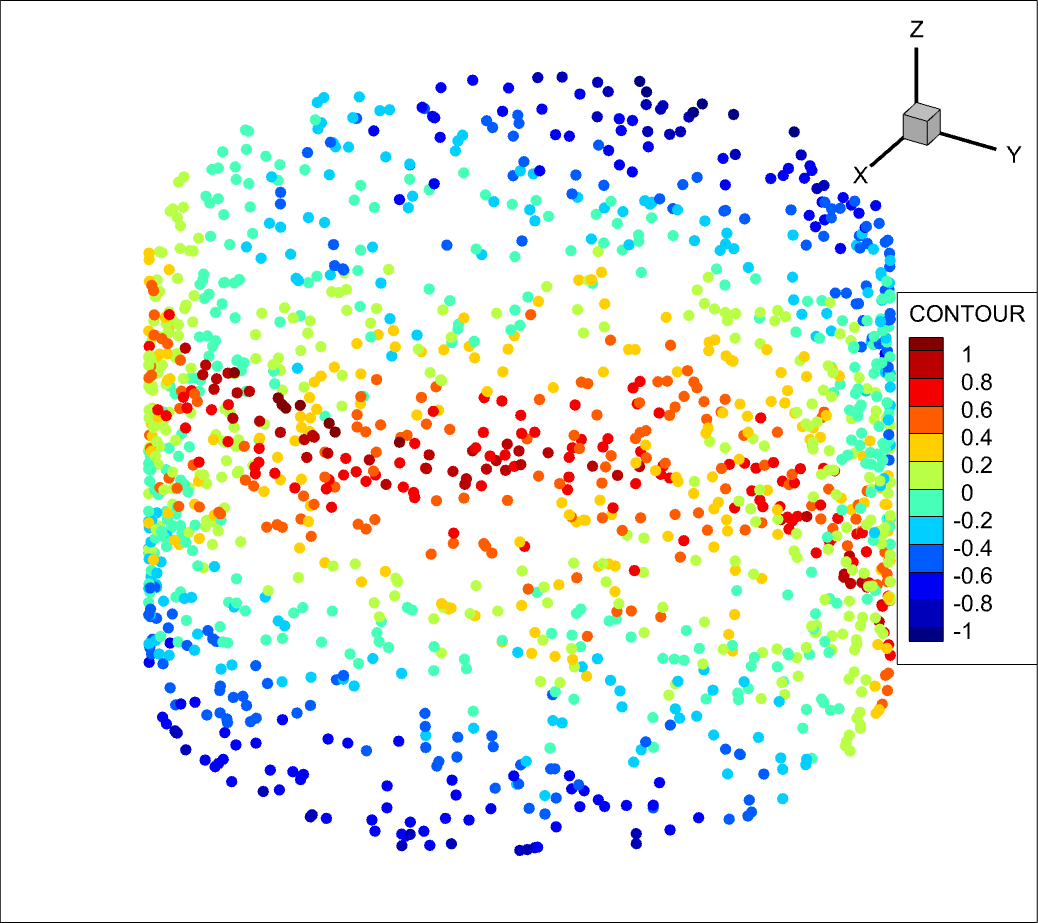}
\includegraphics[width=0.32\textwidth,trim=20 50 20 20,clip]{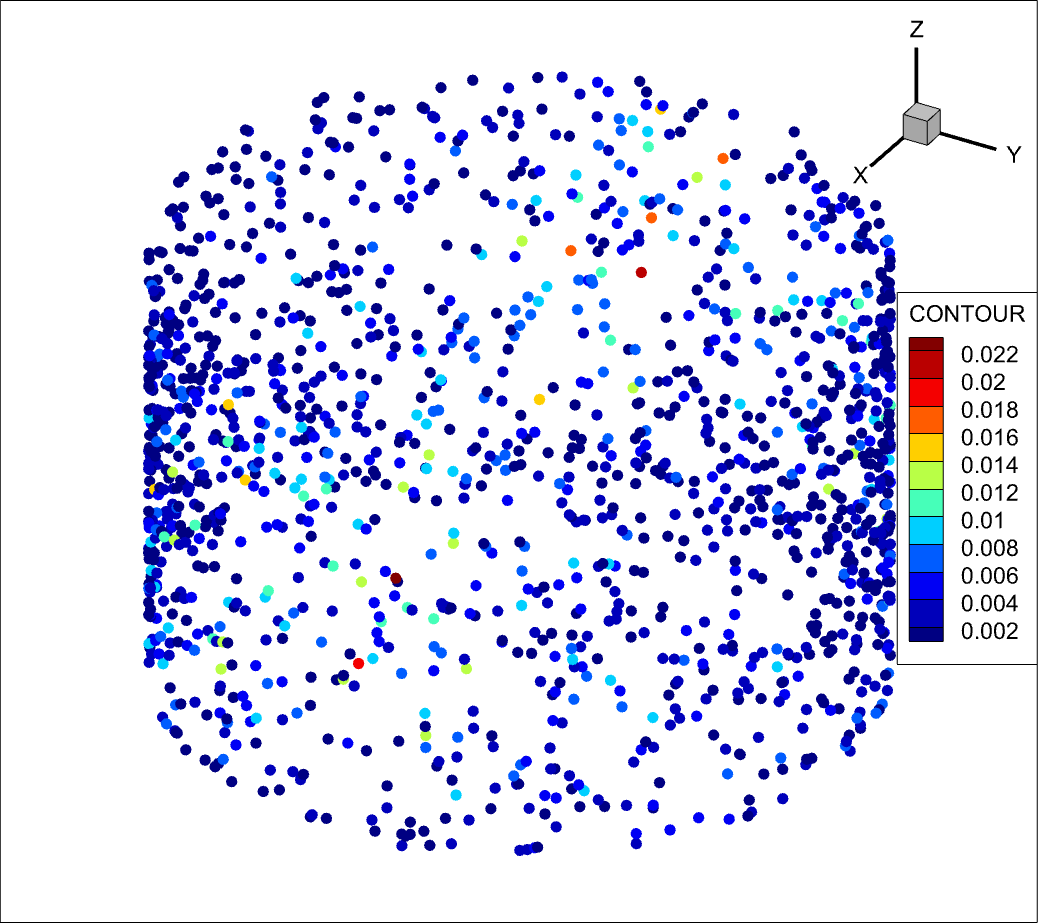}
\caption{Exact solution (left), numerical solution (middle), and absolute error (right) in the $(x_1,x_2,x_7)$ projection.}
\label{fig:S1Fourier_S2_projection}
\end{figure}

\textbf{Example 4} {\bf (Accuracy test on a closed sphere represented by a point cloud).}
In this example, we test the accuracy of the proposed projected Walk on Spheres method on the unit sphere $\mathcal M=S^2\subset\mathbb R^3$, where the manifold is represented only by a uniformly distributed discrete point cloud. More precisely, we take $N_{\mathrm{cloud}}=8000$ points on $S^2$, and solve \eqref{eq:surface_pde} on the closed manifold $\Omega=\mathcal M$ with the following exact solution
$$u_\Omega(x)=P_1(x_3)+0.20\,P_3(x_3)+0.10\,P_5(x_3), \quad x=(x_1,x_2,x_3)\in S^2,$$
where
$$P_1(z)=z,\quad P_3(z)=\frac12(5z^3-3z),\quad P_5(z)=\frac18(63z^5-70z^3+15z)$$
are the Legendre polynomials of degrees $1,3,5$, and the corresponding source term is given by
\[
f_\Omega(x)
=
(\sigma+2)\,P_1(x_3)
+
0.20\,(\sigma+12)\,P_3(x_3)
+
0.10\,(\sigma+30)\,P_5(x_3).
\]

% \begin{wrapfigure}{r}{0.50\textwidth}
% \vspace{-12pt}
% \makebox[\linewidth][r]{\includegraphics[width=0.96\linewidth]{figure/sphere_pointcloud.pdf}}
% \refstepcounter{figure}\label{fig:sphere_pointcloud}
% \vspace{4pt}
% \makebox[\linewidth][r]{\parbox{\linewidth}{\raggedright\small
% \textsc{Figure}~\thefigure. Numerical errors against the number
% of paths for a closed sphere repr- esented
% by a point cloud.}}
% \vspace{-10pt}
% \end{wrapfigure}

\begin{wrapfigure}{r}{0.50\textwidth}
\vspace{-18pt} 
\centering
\includegraphics[width=0.46\textwidth]{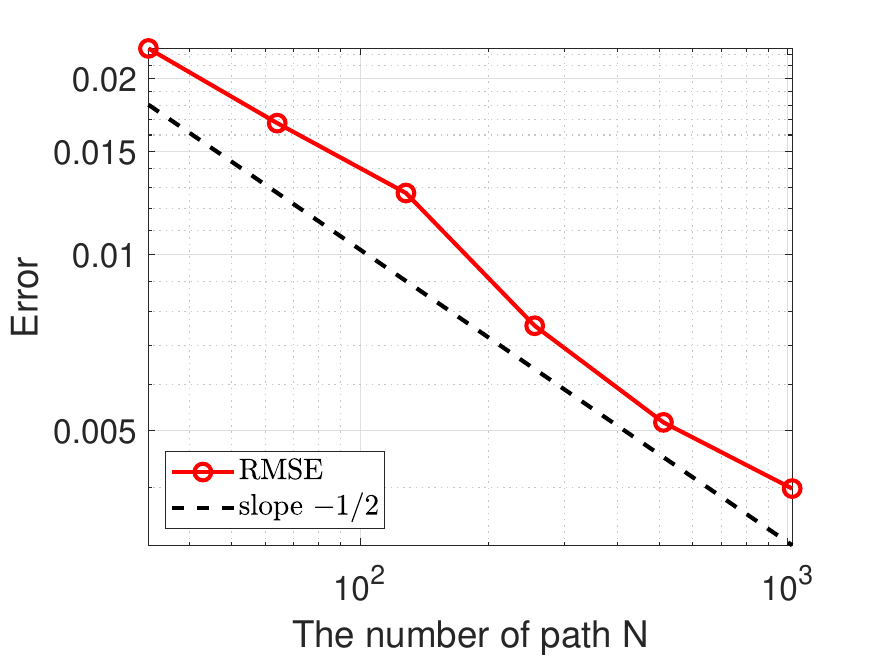}
\vspace{2pt}\caption{Numerical errors against the number of paths for a closed sphere represented by a point cloud.}
\label{fig:sphere_pointcloud}
\vspace{-2pt}
\end{wrapfigure}

The local radius is chosen according to the point cloud radius rule introduced in \eqref{eq:rho_geo_spacing}, where the local neighborhood size $h(y)$ is estimated by the average distance from $y$ to its $k_{\mathrm{nn}}$ nearest neighbors. In this example, we take $k_{\mathrm{nn}}=6$ and $c_h=2.4$, and set the compensation term to zero. The closest-point projection is approximated by nearest neighbor search on the point cloud via a KD tree.
As shown in Fig.~\ref{fig:sphere_pointcloud}, the numerical error decreases steadily as the number of paths increases, and the overall decay is consistent with the reference slope $-1/2$. This indicates that the proposed projected Walk on Spheres method still achieves the expected Monte Carlo convergence rate when the manifold is available only through discrete point-cloud data.

To further illustrate the point-cloud setting, we compare the exact solution, the numerical solution, and the absolute error on the sphere with $N=2000$. As shown in Fig.~\ref{fig:sphere_pointcloud_solution}, the numerical solution agrees closely with the exact solution and accurately captures the overall solution pattern. Moreover, the error remains uniformly small over the point cloud without noticeable concentration or spurious oscillation. These results demonstrate that the proposed MPWoS method remains effective when the manifold is represented only by discrete point-cloud data.

\begin{figure}[htbp]
\hspace*{-0.5cm}
\includegraphics[width=0.32\textwidth,trim=20 50 20 20,clip]{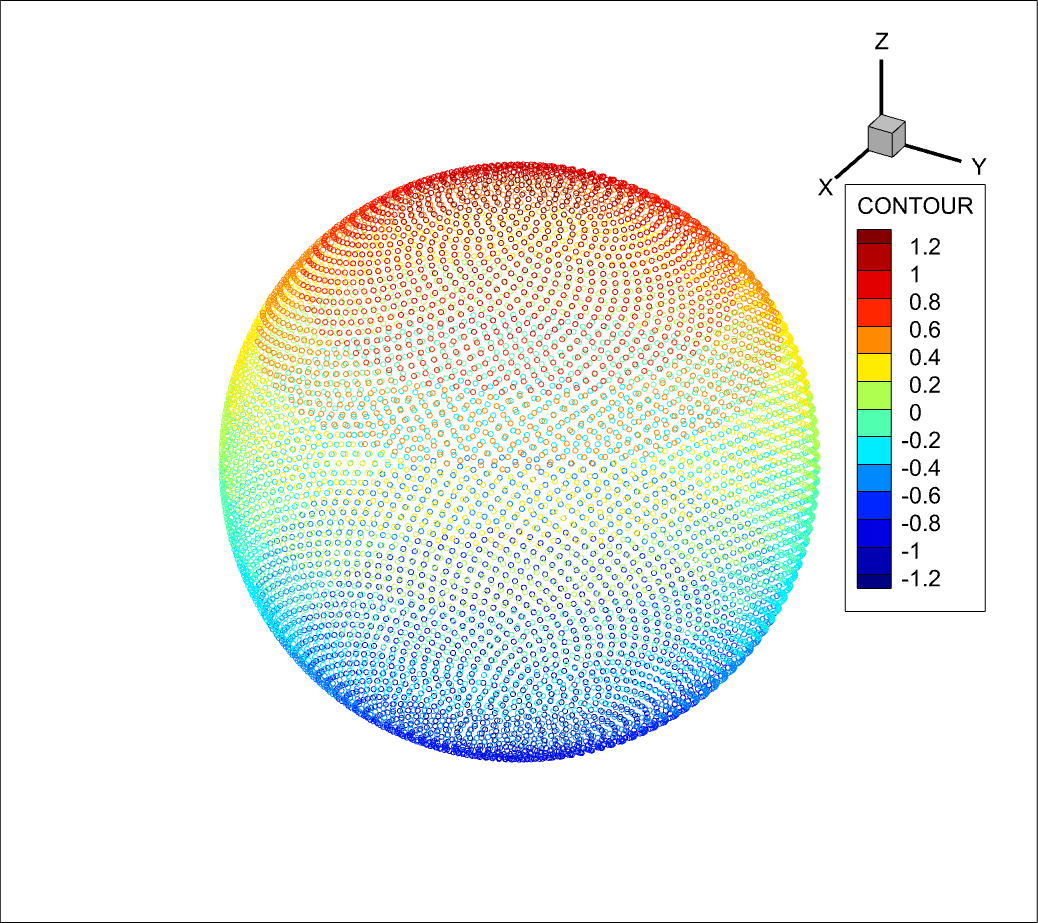}
\includegraphics[width=0.32\textwidth,trim=20 50 20 20,clip]{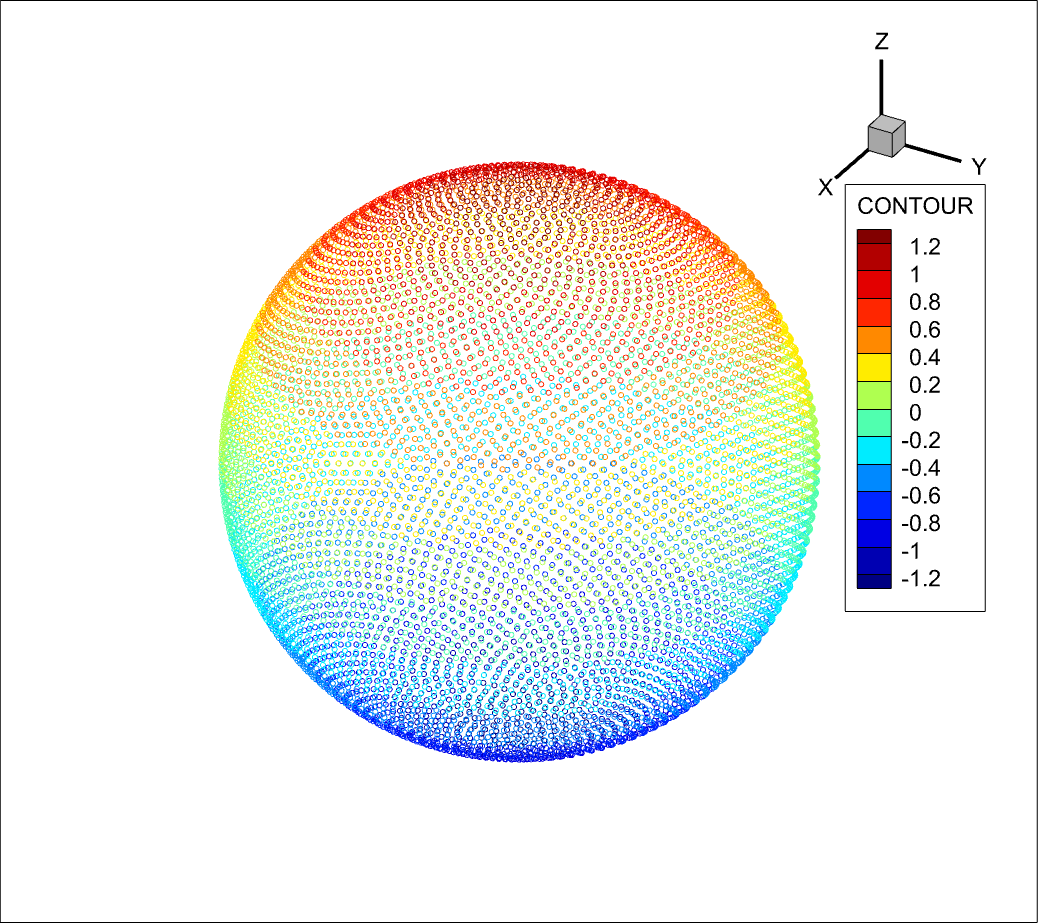}
\includegraphics[width=0.32\textwidth,trim=20 50 20 20,clip]{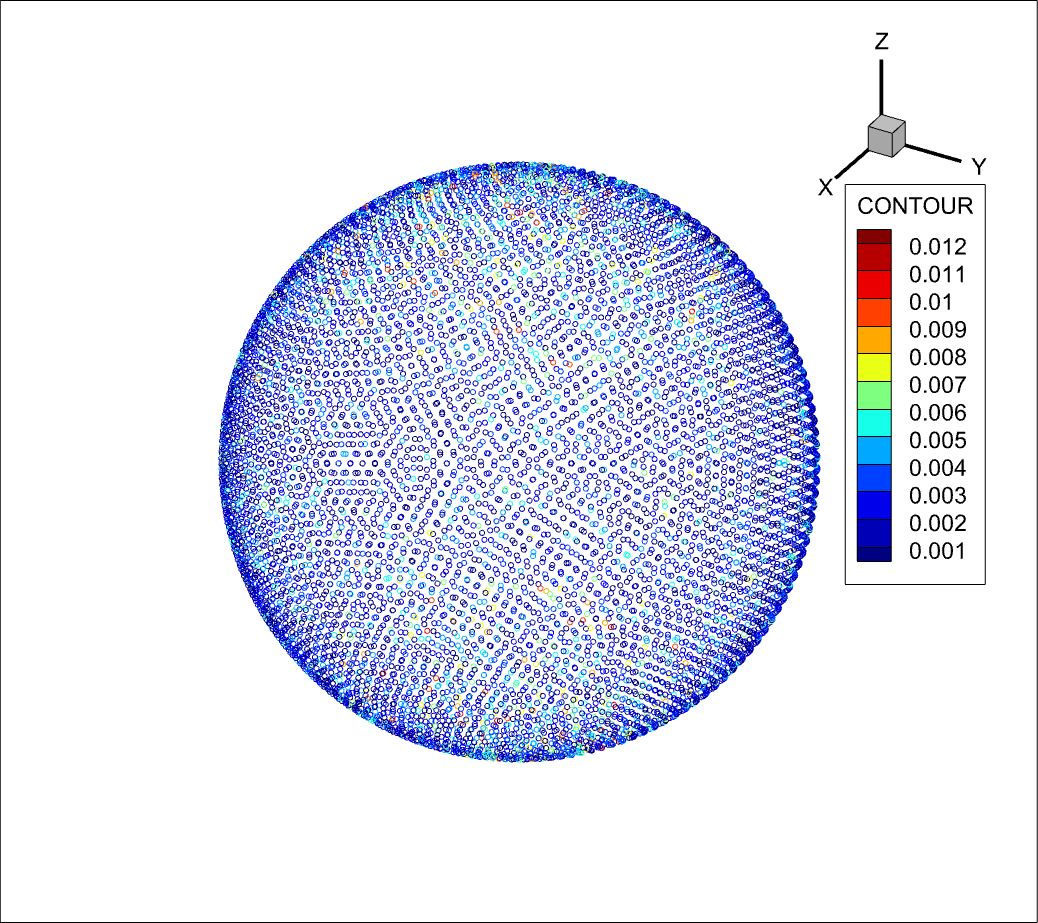}
\caption{Exact solution (left), numerical solution (middle), and absolute error (right) on the closed sphere represented by point cloud.}
\label{fig:sphere_pointcloud_solution}
\end{figure}

\textbf{Example 5} {\bf (Numerical solutions on complex surfaces represented by point clouds).}
In this example, we apply the proposed MPWoS method to two point-cloud surfaces, namely the Stanford bunny and the Stanford dragon, both treated as closed manifolds. The input point clouds are translated to their centroids and normalized by the bounding-box size. We solve \eqref{eq:surface_pde} on the resulting surfaces using the point-cloud radius rule \eqref{eq:rho_geo_spacing}, where the local scale is estimated from the average distance to the $k_{\mathrm{nn}}=10$ nearest neighbors and $c_h=4.0$. The closest-point projection is approximated by nearest-neighbor search. When the original triangular mesh is available, the computed solution is interpolated onto the mesh for visualization.

For both models, we employ the same source construction. After normalization, three representative centers $c_1,c_2,c_3$ are selected as the points with maximal $x$-coordinate, minimal $x$-coordinate, and maximal $z$-coordinate, respectively. We then define
\[
\tilde f_\Omega(x)
=
24\exp\Big(-\frac{|x-c_1|^2}{0.08^2}\Big)
-
18\exp\Big(-\frac{|x-c_2|^2}{0.11^2}\Big)
+
12\exp\Big(-\frac{|x-c_3|^2}{0.07^2}\Big),
\quad x\in\mathcal M,
\]
and then subtract its average over the point cloud, namely,
\[
f_\Omega(x_i)
=
\tilde f_\Omega(x_i)
-
\frac{1}{N_p}\sum_{j=1}^{N_p}\tilde f_\Omega(x_j),
\quad i=1,\dots,N_p,
\]
where $\{x_j\}_{j=1}^{N_p}$ denotes the normalized point cloud. In this way, the source term contains both positive and negative localized excitations while having zero mean on the discrete surface. 
Fig.~\ref{fig:dragon_solution} show the prescribed source terms and the numerical solutions on the bunny and dragon surfaces. The computed solutions remain smooth on the underlying geometry and reflect the spatial distribution of the source term in a stable manner. %In both cases, the numerical error is observed to stay at the level of $10^{-2}$. 
These results indicate that the proposed MPWoS method can be applied effectively to complex surfaces represented only by point-cloud data.

\begin{figure}[ht!]
\hspace*{-0.6cm}

\centering
\includegraphics[width=0.36\textwidth,trim=50 70 5 10,clip]{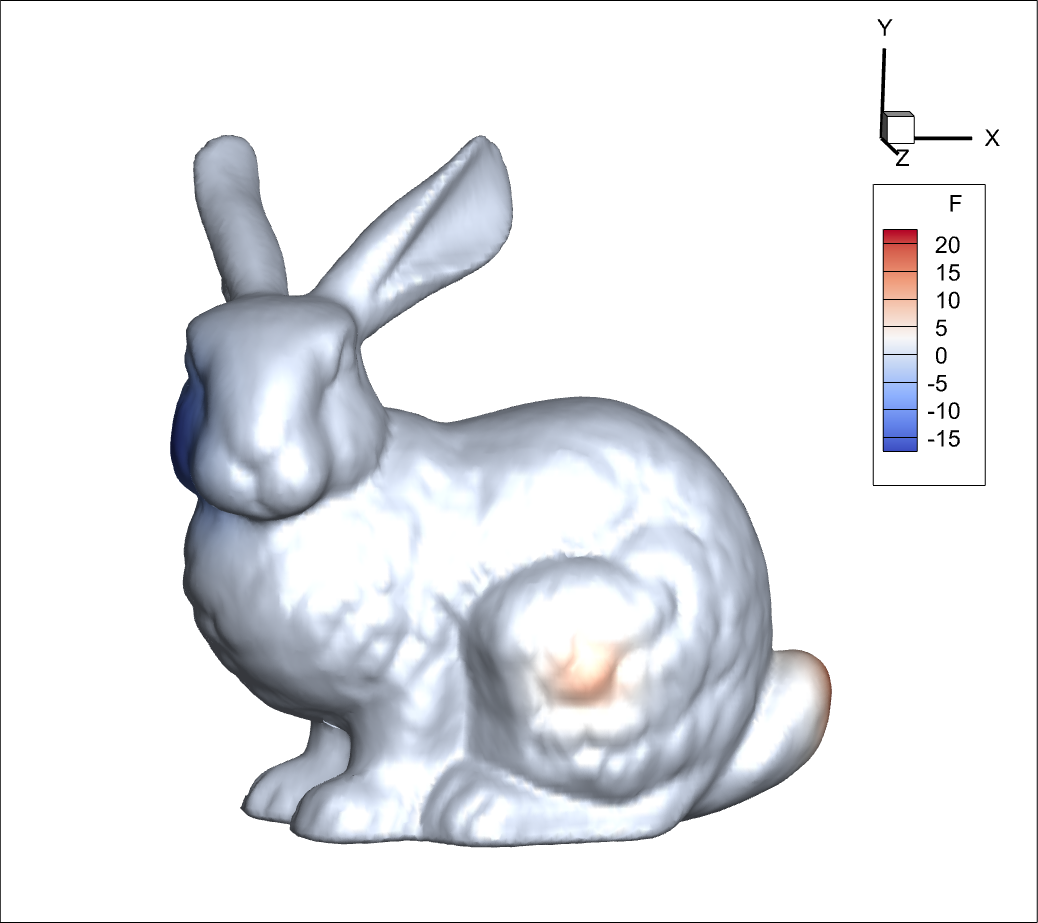}
\includegraphics[width=0.36\textwidth,trim=50 70 5 10,clip]{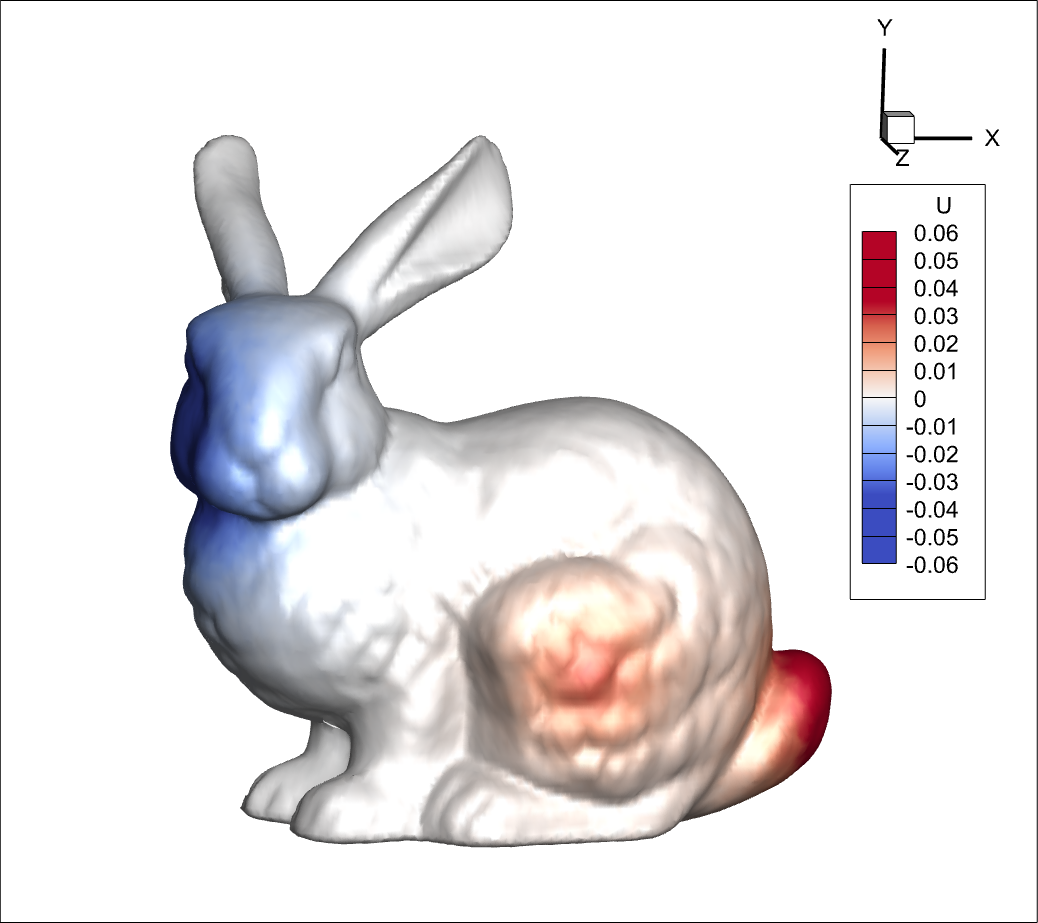}

\vspace{-2pt}
\includegraphics[width=0.36\textwidth,trim=50 70 5 10,clip]{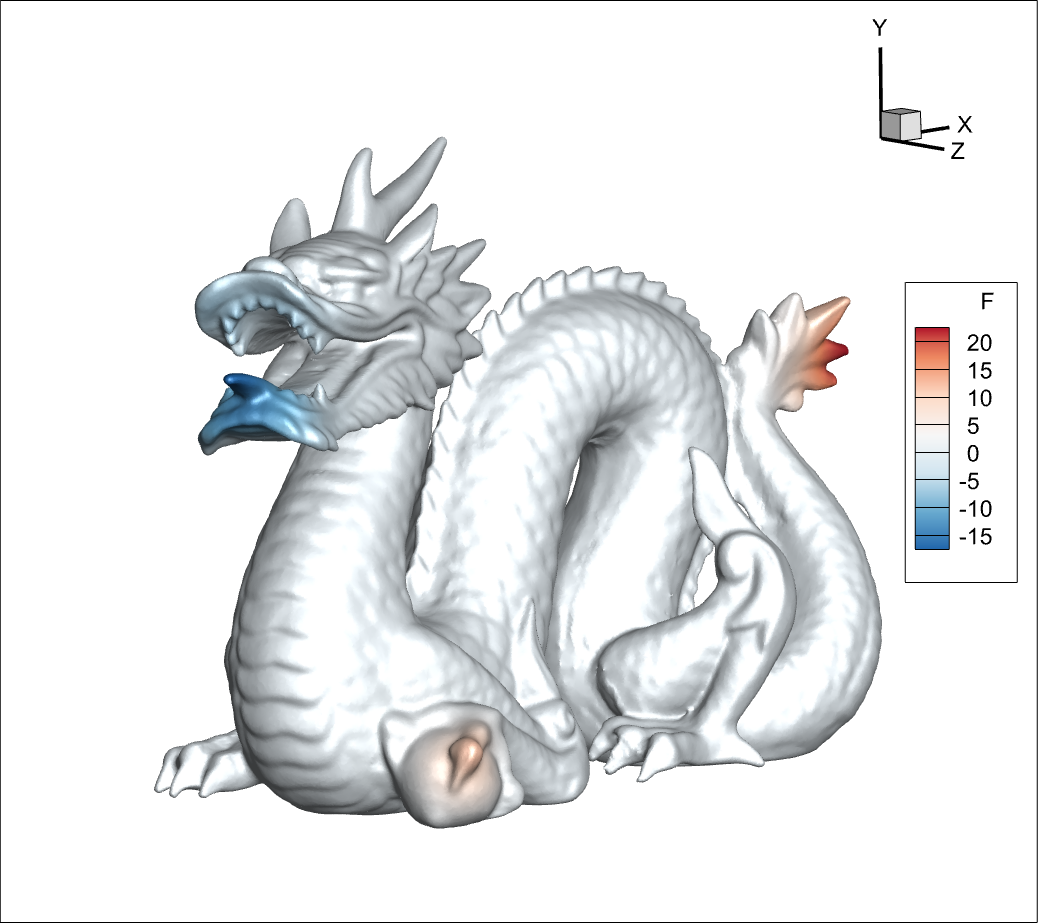}
\includegraphics[width=0.36\textwidth,trim=50 70 5 10,clip]{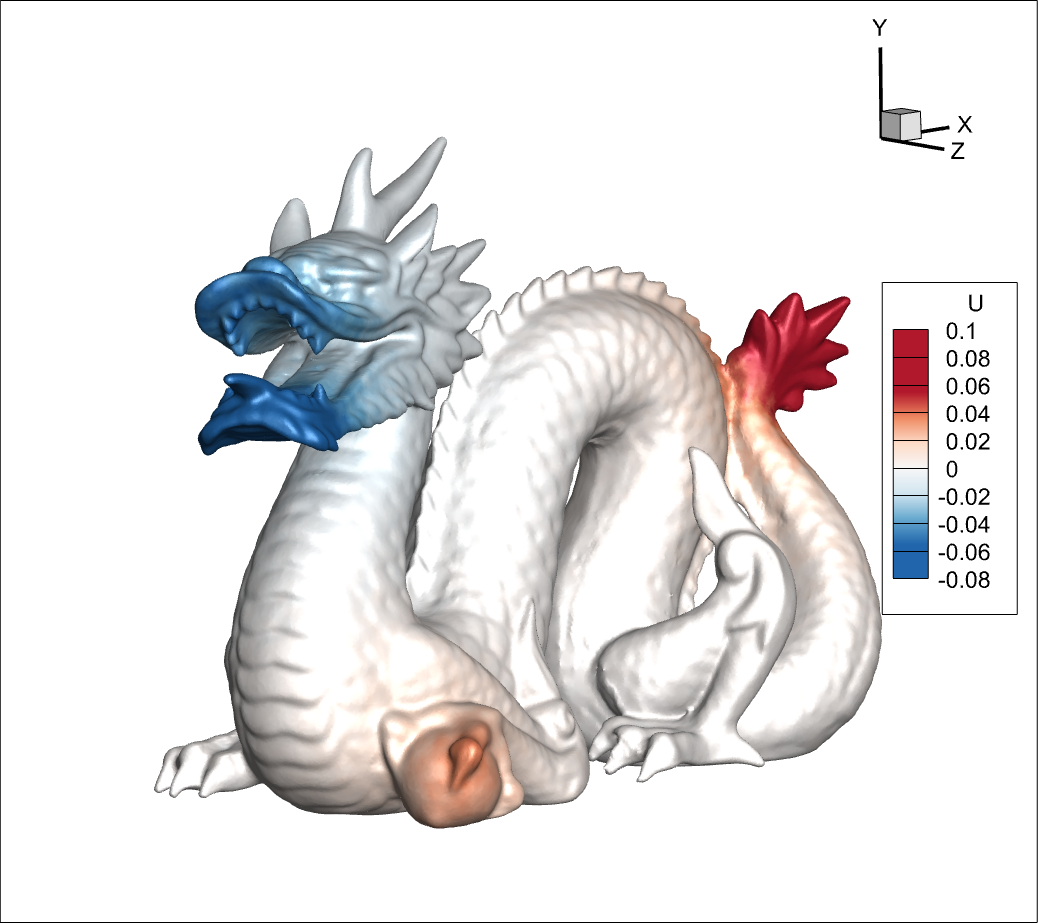}

%\includegraphics[width=0.36\textwidth,trim=20 50 5 5,clip]{figure/bunny_source.png}
%\includegraphics[width=0.36\textwidth,trim=20 50 5 5,clip]{figure/bunny_solution.png}
%%\caption{Source term (left) and numerical solution (right) on the Stanford bunny.}
%%\label{fig:bunny_solution}
%%\end{figure}
%%
%%\begin{figure}[ht!]
%\hspace*{0.6cm}
%\includegraphics[width=0.36\textwidth,trim=20 50 5 5,clip]{figure/dragon_source.png}
%\includegraphics[width=0.36\textwidth,trim=20 50 5 5,clip]{figure/dragon_solution.png}
\caption{Source term (left) and numerical solution (right) on the Stanford bunny (top) and Stanford dragon (bottom).}
\label{fig:dragon_solution}
\end{figure}

\section{Conclusion}
In this paper, we proposed a modified projected Walk on Spheres method for solving screened Poisson equations on embedded manifolds. The method is based on local sampling in the ambient Euclidean space together with closest-point projection, and applies to both manifolds with boundary and closed manifolds, including high dimensional embedded manifolds represented in parametric, implicit, or point-cloud form. We derived a local probabilistic representation and the associated projected recursion, and established mean square error estimates for the proposed Monte Carlo method in both the case of manifolds with boundary and the closed manifold case. The numerical experiments confirmed the expected half order convergence behavior of the proposed method and illustrated the effectiveness of the compensation term in reducing the mismatch error. They also showed that the method remains effective for manifolds of different dimensions and geometric complexity, including higher dimensional embedded manifolds and complex point cloud geometries such as the bunny and dragon. These results indicate that the proposed MPWoS method provides an effective stochastic solver for screened Poisson equations on manifolds.

% \vspace{18pt}

% \noindent{\bf Declarations}

% \begin{itemize}
% \item {\bf Availability of data and materials:}  The datasets generated during and/or analysed during the current study are available from the corresponding author on reasonable request. %The data that support the findings of this study are available from the corresponding author upon reasonable request.
% \item {\bf Authors' contributions:} All authors contributed to this study. The computations and the first draft were prepared by the first and second authors. All authors read and approved the final manuscript.
% \item {\bf Conflict of interest statement:}   We have no conflicts of interest to disclose.
% \end{itemize}

\end{document}